\documentclass[12pt, letterpaper]{amsart}
\usepackage[dvipsnames,table,xcdraw]{xcolor}
\usepackage{amsfonts,amsthm,mathtools}
\usepackage[all,cmtip]{xy}
\usepackage{fullpage}
\usepackage{bm}
\usepackage{amsmath}
\usepackage{amssymb}
\usepackage{enumerate}
\usepackage{tikz}
\usepackage{hyperref}
\usepackage{cleveref}
\usepackage{verbatim}
\usepackage{cancel}
\usepackage{float}
\usepackage{todonotes}
\usepackage{kotex}
\usepackage{longtable}

\usepackage{graphicx} 

\newcommand{\Sym}{\operatorname{Sym}}
\newcommand{\Pic}{\operatorname{Pic}}
\newcommand{\gon}{\operatorname{gon}}

\newcommand{\Q}{\mathbb{Q}}

\newcommand{\C}{\mathbb{C}}

\newcommand{\N}{\mathbb{N}}
\newcommand{\Z}{\mathbb{Z}}
\newcommand{\F}{\mathbb{F}}
\newcommand{\PP}{\mathbb{P}}

\newcommand{\SL}{\operatorname{SL}}

\newtheorem{lem}{Lemma}[section]
\newtheorem{thm}[lem]{Theorem}
\Crefrangeformat{thm}{Theorems #3#1#4-#5#2#6}
\newtheorem{proposition}[lem]{Proposition}
\newtheorem{prop}[lem]{Proposition}
\Crefrangeformat{prop}{Propositions #3#1#4-#5#2#6}
\newtheorem{cor}[lem]{Corollary}

\theoremstyle{definition}
\newtheorem{remark}[lem]{Remark}
\newtheorem{definition}[lem]{Definition}

\crefname{lem}{Lemma}{Lemmas}
\Crefname{lem}{Lemma}{Lemmas}

\newcommand{\github}[1]{\href{https://github.com/nt-lib/density-degree-5-X0/blob/v0.1.1/#1}{\path{#1}}}
\newcommand{\lmfdbec}[3]{\href{https://www.lmfdb.org/EllipticCurve/Q/#1/#2/#3}{#1.#2#3}}
\newcommand{\maarten}[1]{{\color{ForestGreen} \sf $\diamondsuit\diamondsuit\diamondsuit$ Maarten: [#1]}}
\newcommand{\petar}[1]{{\color{purple} \sf $\diamondsuit\diamondsuit\diamondsuit$ Petar: [#1]}}
\newcommand{\daeyeol}[1]{{\color{blue} \sf $\diamondsuit\diamondsuit\diamondsuit$ Daeyeol: [#1]}}

\title{Modular curves $X_0(N)$ of density degree $5$}

\begin{document}

\begin{abstract}
    We determine all modular curves $X_0(N)$ with minimum density degree $5$, i.e. all curves $X_0(N)$ with infinitely many points of degree $5$ and only finitely many points of degree $d\leq4$. As a consequence, the problem of determining all curves $X_0(N)$ with infinitely many points of degree $5$ remains open for only $30$ levels $N$.
\end{abstract}


\author{\sc Maarten Derickx}
\address{Maarten Derickx\\
University of Zagreb\\  
Bijeni\v{c}ka Cesta 30 \\
10000 Zagreb\\
Croatia}
\email{maarten@mderickx.nl}
\urladdr{http://www.maartenderickx.nl/}

\author{\sc Wontae Hwang}
\address{Wontae Hwang\\ Department of Mathematics, and Institute of Pure and Applied Mathematics, Jeonbuk National University \\
Baekje-daero, Deokjin-gu, Jeonju-si, Jeollabuk-do, 54896 \\
South Korea}
\email{hwangwon@jbnu.ac.kr}

\author{\sc Daeyeol Jeon}
\address{Daeyeol Jeon \\ Department of Mathematics Education \\ Kongju National University \\ Gongju, 32588 South Korea}
\email{dyjeon@kongju.ac.kr}

\author{\sc Petar Orli\'c}
\address{Petar Orli\'c \\
University of Zagreb\\  
Bijeni\v{c}ka Cesta 30 \\
10000 Zagreb\\
Croatia}
\email{petar.orlic@math.hr}

\maketitle

\section{Introduction}

Let $C$ be a smooth curve defined over a number field $K$. 
Determining whether the set of points of degree $\le d$ on $C$ is finite or infinite is an important problem in arithmetic geometry.
Here, a point of degree $d$ means a closed point $P$ such that $[K(P):K]=d$.
Faltings' Theorem solves the base case $d=1$.

\begin{thm}[Faltings' Theorem]
Let $K$ be a number field and let $C$ be a non-singular curve defined over $K$ of genus $g\geq2$. Then the set $C(K)$ is finite.
\end{thm}

Therefore, the curve $C$ has infinitely many $K$-rational points if and only if $C$ is isomorphic to $\mathbb{P}^1$ ($g=0$) or $C$ is an elliptic curve ($g=1$) with positive $K$-rank. We now move on to the case $d=2$.

\begin{thm}[{Harris, Silverman: \cite[Corollary 3]{HarrisSilverman91}}]
    Let $K$ be a number field, and let $C/K$ be a curve of genus at least $2$. Assume that $C$ is neither hyperelliptic nor bielliptic. Then the set of quadratic points on $C/K$ is finite.
\end{thm}


\begin{definition}\cite[Page 1]{KV2025}
    Let $C$ be a curve defined over a number field $K$. The {\it density degree set} $\delta(C/K)$ is the set of integers $d$ for which the collection of closed points of degree $d$ on $C$ is infinite. The {\it minimum density degree}\footnote{In previous versions of \cite{KV2025}, the minimum density degree was called \textit{the arithmetic degree of irrationality} $\textup{a.irr}_K C$.} $\min(\delta(C/K))$ is the smallest integer $d\geq1$ such that $C$ has infinitely many closed points of degree $d$ over $K$, or equivalently:
    $$\min(\delta(C/K)) =\min\left\{d 
\in \N \mid \#\left\{\cup_{[F:K]=d} C(F)\right\}=\infty\right\}.$$
    We also define the {\it potential density degree set} $\wp(C/K):=\bigcup_{[L:K]<\infty}\delta(C/L)$, and analogously the {\it minimum potential density degree}
    $$\min(\wp(C/K)) :=\min \bigcup_{[L:K]<\infty}\delta(C/L).$$
\end{definition}

Abramovich and Harris \cite{AbramovichHarris91} gave a conjecture
$$\wp(C/K)\leq d \iff C \textup{ admits a map of degree} \leq d \textup{ to } \mathbb{P}^1 \textup{ or an elliptic curve}$$
which they proved for $d=2,3$. However, Debarre and Fahlaoui constructed counterexamples to the conjecture for $d\geq4$ \cite{DebarreFahlaoui}.

A first step in determining the modular curves $X_0(N)$ of minimum density degree $5$ is to show that for $N$ large enough, the modular curve $X_0(N)$ has only finitely many points of degree $\leq 5$. Or equivalently, for $N$ large enough, the minimum density degree of $X_0(N)$ is strictly larger than $5$.

A key ingredient to do this is the following theorem by Kadets and Vogt, that allows one to study the minimum density degree of a curve in terms of maps to lower-genus curves when the genus of the given curve is large enough.
\begin{thm}[{\cite[Theorem 1.3]{KV2025}}]\label{KVthm}
Suppose that $C$ is a nice\footnote{nice here means smooth, projective, and geometrically integral} curve of genus $g$ over a number field $K$ and $\min(\delta (C/K))=d.$ Let $m:=\lceil{d/2}\rceil -1$ and let $\epsilon := 3d-1-6m < 6.$ Then one of the following holds:
\vskip 0.1in
(1) There exists a non-constant morphism of curves $\phi \colon C \rightarrow Y$ of degree at least $2$ such that $d = \min(\delta(Y/K)) \cdot \deg \phi.$
\vskip 0.1in
(2) $g \leq \max \left(\frac{d(d-1)}{2} + 1, 3m (m-1)+m \epsilon \right).$
\end{thm}
\begin{cor}\label{new_cond_quintic}
If $\min(\delta(X_0(N)/\Q))=5,$ and $g(X_0(N)) \geq 12,$ then $X_0(N)$ admits a $\Q$-rational map of degree $5$ to either $\PP^1$ or an elliptic curve $E$ of positive rank over $\Q$.
\end{cor}
\begin{proof}
Since $g(X_0(N)) \geq 12$, case (2) of Theorem \ref{KVthm} cannot occur. This means that there exists a non-constant morphism of curves $\phi \colon X_0(N) \rightarrow Y$ of degree at least $2$ over $\mathbb{Q}$ with the property that $5= \min(\delta(Y/\Q)) \cdot \deg \phi.$
It follows that $\deg \phi = 5$ and $\min(\delta(Y/\Q))=1,$ and hence, $X_0(N)$ admits a $\mathbb{Q}$-rational map of degree $5$ to either $\PP^1$ or an elliptic curve of positive rank over $\Q$, as desired.
\end{proof}

Determining degree $d$ points over $\Q$ on modular curves is especially interesting because non-cuspidal points on modular curves represent isomorphism classes of elliptic curves. For example, points on curves $X_1(M,N)$ represent isomorphism classes of elliptic curves with torsion group $\Z/M\Z\times\Z/N\Z$, and points on curves $X_0(N)$ represent isomorphism classes of elliptic curves together with a cyclic $N$-isogeny.

Regarding the curves $C=X_1(M,N)$ and $K=\Q$, all cases when $C$ has infinitely many points of degree $d\leq6$ were determined by Mazur \cite{mazur77} (for $d=1$), Kenku, Momose, and Kamienny \cite{KM88, kamienny92} (for $d=2$), Jeon, Kim, and Schweizer \cite{JeonKimSchweizer04} (for $d=3$), Jeon, Kim, and Park \cite{JeonKimPark06} (for $d=4$), and Derickx and Sutherland \cite{DerickxSutherland17} (for $d=5,6$). Additionally, Derickx and van Hoeij \cite{derickxVH} determined all curves $X_1(N)$ which have infinitely many points of degree $d=7,8$ and Jeon determined all trielliptic \cite{Jeon2022} and tetraelliptic \cite{Jeon2025} curves $X_1(N)$ over $\Q$.

In this paper, we will study the curve $C=X_0(N)$ and $K=\Q$. The curve $X_0(N)$ has infinitely many rational points if and only if $g(X_0(N))=0$. This was proved by Mazur \cite{mazur78} and Kenku \cite{kenku1979, kenku1980_1, kenku1980_2, kenku1981}. 

Ogg \cite{Ogg74} determined all hyperelliptic curves $X_0(N)$, Bars \cite{Bars99} determined all bielliptic curves $X_0(N)$, as well as all curves $X_0(N)$ with infinitely many quadratic points, and Jeon \cite{Jeon2021} determined all curves $X_0(N)$ with infinitely many cubic points. All curves $X_0(N)$ with infinitely many quartic points were determined independently (and using different methods) by Hwang and Jeon \cite{Hwang2023}, and Derickx and Orlić \cite{DerickxOrlic23}.

\begin{thm}[Bars \cite{Bars99}]\label{thmquadratic}
    The modular curve $X_0(N)$ has infinitely many points of degree $2$ over $\Q$ if and only if
    $$N\in\{1-33,35-37,39-41,43,46-50,53,59,61,65,71,79,83,89,101,131\}.$$
\end{thm} 

\begin{thm}[Jeon \cite{Jeon2021}]\label{thmcubic}
    The modular curve $X_0(N)$ has infinitely many points of degree $3$ over $\Q$ if and only if
    $$N\in\{1-29,31,32,34,36,37,43,45,49,50,54,64,81\}.$$
\end{thm}

\begin{thm}[Hwang, Jeon; Derickx, Orlić \cite{Hwang2023, DerickxOrlic23}]\label{quarticthm}
    The modular curve $X_0(N)$ has infinitely many points of degree $4$ over $\Q$ if and only if
    \begin{align*}
        N\in\{&1-75,77-83,85-89,91,92,94-96,98-101,103,104,107,111,\\
        &118,119,121,123,125,128,131,141-143,145,155,159,167,191\}.
    \end{align*}
\end{thm}

From these results, it is easy to determine the modular curves of the form $X_0(N)$ with minimum density degree $1$, $2$, $3$ and $4$. For example, $X_0(N)$ has minimum density degree $4$ if and only if $N$ occurs in the set of \Cref{quarticthm} but not in the sets of Theorems \ref{thmquadratic} and \ref{thmcubic}.
Motivated by the recent work of Kadets and Vogt \cite{KV2025} on density degrees, we determine all curves $X_0(N)$ with minimum density degree equal to $5$. Our main result is the following theorem:
\begin{thm}\label{densitydeg5thm}
    The modular curve $X_0(N)$ has minimum density degree $\min(\delta(X_0(N)/\Q))$ equal to $5$ if and only if $N=109$.
\end{thm}

As a guiding geometric intuition, the exceptional nature of the level $N=109$ can be attributed to the fact that
$X_0(109)$ is the unique modular curve $X_0(N)$ with $\mathbb{Q}$-gonality~$5$,
that is, it admits a non-constant $\mathbb{Q}$-rational morphism of minimal degree~$5$
to $\mathbb{P}^1$. 
Moreover, this exceptional nature does not arise from degree $5$ $\mathbb{Q}$-rational maps to elliptic curves of positive rank, since no modular curve $X_0(N)$ admits such a map; see \Cref{thm:pos_rank_penta}.

The curve $X_0(109)$ being the unique curve of density degree $5$ seems to fit into a broader pattern.
Indeed, density degrees appear to be much rarer in odd degree:
$X_0(54), X_0(64)$, and $X_0(81)$ are the only curves $X_0(N)$
whose minimum density degree is equal to $3$ by \Cref{thmcubic}.
By contrast, there are many more curves $X_0(N)$ of density degrees $2$ and $4$,
see \Cref{thmquadratic} and \Cref{quarticthm}.

A partial geometric explanation of this phenomenon is as follows. It is much easier to obtain morphisms from $X_0(N)$ of even degree because all Atkin-Lehner involutions $w_d$ are of degree $2$ and the degeneracy maps $X_0(N)\to X_0(M)$ are often of even degree. We often found that maps of minimal degree to $\PP^1$ (\cite{NajmanOrlic22}) and elliptic curves (\cite{Hwang2023, DerickxOrlic23}) factor through Atkin-Lehner quotients, making them of even degree.

Theorem~\ref{densitydeg5thm} is an important step in determining all curves $X_0(N)$ with infinitely many quintic points. In fact, if a curve has infinitely many quintic points, then it has to have minimum density degree $\leq 5$. And our result completes the classification of all modular curves of minimum density degree $\leq 5$. 
It is important to note that this does not amount to a complete
classification of all curves with infinitely many quintic points:
while every curve of minimum density degree $5$ necessarily has infinitely
many quintic points, the latter phenomenon may also occur for curves
of a smaller minimum density degree.
Therefore, in order to determine all curves $X_0(N)$ with infinitely many quintic points, it suffices to determine for the explicit list $X_0(N)$ of minimum density degree $\leq 4$ whether they have infinitely many quintic points. With some extra work, we are able to prove the following:

\begin{thm}\label{quinticcor}
    The modular curve $X_0(N)$ has infinitely many points of degree $5$ for \[N\in\{1-45,47-58,61,63,64,65,67,68,72,73,75,80,81,91,109,121,125\}.\] The modular curve $X_0(N)$ has only finitely many points of degree $5$ for $N\geq192$ and
    \begin{align*}
        N\in\{&46,59,60,62,66,69,70,71,76,78,84,87,90,93,94,95,97,102,104,105,\\
        &106,108,110,112-117,119,120,122,124,126,127,129,130,132-140,\\
        &144,146-154,156,157,158,160-166,168-190\}.
    \end{align*}
\end{thm}

\begin{remark}\label{rem:remain-cases} With the above theorem, there are 30 levels $N$ for which we do not know yet if there are infinitely many quintic points. These levels are
\begin{align*}
\{&74,77,79,82,83,85,86,88,89,92,96,98,99,100,101,103,107,\\
        &111,118,123,128,131,141,142,143,145,155,159,167,191\}.
\end{align*}
The remaining 30 levels are technically difficult for the following reasons.
For each of these levels $N$, the modular curve $X_0(N)$ already has infinitely many points of degree at most $4$.
Equivalently, there exists a $\mathbb{Q}$-rational map of degree $\le 4$ from $X_0(N)$ to either $\mathbb{P}^1$
or to a positive rank elliptic curve.
On the other hand, none of these curves admits a $\mathbb{Q}$-rational map of degree $5$
to $\mathbb{P}^1$ or to a positive rank elliptic curve; see \Cref{thm:pos_rank_penta}.

By Theorem~\ref{translate_of_abelian_variety_thm}, if $X_0(N)$ has infinitely many quintic points, then
$W_5^0(X_0(N))$ must contain a translate of a positive rank abelian subvariety
of $\Pic^0(X_0(N))$ coming from a quintic point.
However, for all such $N$, $\Pic^0(X_0(N))$ already contains positive rank
abelian subvarieties, and the inclusions $W_d^0(X_0(N)) \subset W_5^0(X_0(N))$
hold for $d<5$.
Thus it is difficult to distinguish whether a given translate in $W_5^0(X_0(N))$
arises from genuinely quintic points or from points of lower degree.
This will be discussed further in \Cref{deg5pointssection}.

\end{remark}

The outline of the paper is as follows: 
\begin{itemize}
    \item In \Cref{section_preliminaries} we present the previously known results that will be used in later sections of the paper. 
    \item In \Cref{positive_rank_pentaelliptic} we prove that there are no curves $X_0(N)$ that admit a degree $5$ rational morphism to a positive rank elliptic curve using the techniques from \cite{DerickxOrlic23}. This is an essential result in proving \Cref{densitydeg5thm}. Combined with the $\Q$-gonalities computed in \cite{NajmanOrlic22} and \Cref{new_cond_quintic}, this limits the number of levels where $\min(\delta(X_0(N)/\Q))=5$ to a set of only 17 values.
    \item In \Cref{AVtranslates_section} we use four different techniques     (\Cref{jacobian_rank0_prop}, \Cref{no_AV_of_dim_12}, \Cref{DFcurves}, and
    \Cref{prop112117}) to show that $X_0(N)$ has only finitely many degree $5$ points. These $4$ techniques allow us to prove \Cref{densitydeg5thm} in \Cref{main_results_section}.

    After that, we move on to prove \Cref{quinticcor}.
    \item In \Cref{cs_section} we use the Castelnuovo-Severi inequality to prove that for some levels $N$ the curve $X_0(N)$ does not admit a degree $5$ morphism to $\PP^1$ or an elliptic curve. This is needed for the proof of \Cref{quinticcor} in \Cref{deg5pointssection}.
    \item In \Cref{infinitelydeg5section} we prove that for the levels in the first list of \Cref{quinticcor}, there is a degree $5$ rational morphism to $\PP^1$.
    \item In \Cref{deg5pointssection} we prove \Cref{quinticcor} and give a short discussion concerning the remaining $30$ levels $N$.
\end{itemize} 

\section*{Reproducibility and data availability}
The results in this paper depend on computations carried out in Magma (V2.29-4) \cite{magma}, SageMath (10.5) \cite{sagemath}, Maple (2024) \cite{maple}, and PARI/GP (2.15.5) \cite{PARIGP}. 

We also used the packages QuadraticPoints and MD Sage (v0.1.0). MD Sage is a SageMath package and can be found at:
\begin{center}
    \url{https://github.com/koffie/mdsage/tree/v0.1.0}
\end{center}
QuadraticPoints is a Magma package which can be found at:
\begin{center}
    \url{https://github.com/TimoKellerMath/QuadraticPoints/}
\end{center}
This package contains the Magma code accompanying the article \cite{akmnov}.
QuadraticPoints does not have versioned releases; we used the version with commit hash: 
\begin{center}
\href{https://github.com/TimoKellerMath/QuadraticPoints/tree/267352126eb18a4737eac4484191bd13a864fe8b}{267352126eb18a4737eac4484191bd13a864fe8b}.
\end{center}

The code for verifying all computations in this paper can be found at:
\begin{center}
\url{https://github.com/nt-lib/density-degree-5-X0}
\end{center}
Filenames in this article refer to files in the above GitHub repository.

The Magma and PARI/GP computations have been run on a server at the University of Zagreb with an Intel Xeon W-2133 CPU @ 3.60GHz with 6 cores and 64GB of RAM running Ubuntu 18.04.6 LTS. With the Magma computations taking roughly 6 minutes and 485 MB of ram, while the PARI/GP computations completed within a few seconds. 

The Sagemath computation were run on MacBook Pro with an Apple M3 chip with 8 cores (4 performance and 4 efficiency) and 8 GB of memory. We assume that the computations were executed on the performance cores, which run at 4.05GHz but were not able to verify this. These computation took roughly 12 minutes and 900 MB of ram.

The Maple computations were performed at Kongju National University on a workstation with an Intel Core i9-10920X CPU @ 3.50GHz and 256GB of RAM, taking approximately a few minutes.

More details on how to reproduce the computation can be found in the \texttt{README.md} file of the aforementioned \texttt{density-degree-5-X0} repository.

\section*{Acknowledgments}

We thank Filip Najman for his useful comments on the paper. We also thank the referee for a careful reading of our manuscript and for the valuable and constructive comments.

The first and fourth authors were supported by the Croatian Science Foundation under the project no. IP-2022-10-5008 and by the project ``Implementation of cutting-edge research and its application as part of the Scientific Center of Excellence for Quantum and Complex Systems, and Representations of Lie Algebras“, PK.1.1.10.0004, European Union, European Regional Development Fund.

The second author was supported by Global-Learning $\&$ Academic research institution for Master’s·PhD students, and Postdocs (LAMP) Program of the National Research Foundation of Korea (NRF) grant funded by the Ministry of Education(No. RS-2024-00443714). 

The third author was supported by Basic Science Research Program through the National Research Foundation of Korea (NRF) funded by the Ministry of Education (No. 2022R1A2C1010487).

\section{Preliminaries}\label{section_preliminaries}

For a curve $C$, a number field $K$ and a positive integer $d$, let
\begin{align*}
C_d(K)&=\{ P \in C(\overline K)\colon [K(P):K]\leq d \},\\
C_d'(K)&=\{ P \in C(\overline K)\colon [K(P):K]=d \}.
\end{align*}

With this definition, one sees that there are infinitely many points of degree $d$ on $C$ if and only if $\#C_d'(K) = \infty$.

Let $C$ be a curve over $\C$. We define \begin{align}W_d^r(C):=\{D\in\Pic^d(C): \ell(D)>r\}. \label{eq:Wdr_over_C}\end{align} We will also sometimes need to use $W_d^0$ in a more general setting. Let $R$ be a commutative ring and let $C$ be a smooth projective curve over $R$ with geometrically irreducible fibers. For simplicity, we also assume that $C(R) \neq \emptyset$ so that the Picard functor is representable. Then $\Sym^d(C)$ will denote the $d$-th symmetric power of $C$ on $R$. We define \[\phi_d:\Sym^d(C)\to\Pic^d (C)\] to be the morphism that sends an effective divisor of degree $d$ to its associated line bundle. Furthermore, we use \begin{align}W_d^0 C := \textup{Im } \phi_d \subset \Pic^d(C)\label{eq:Wd0}\end{align} to denote the scheme theoretic image of $\phi_d$. Note that this definition matches with the one over $\C$, since a line bundle has a non-trivial section if and only if it comes from an effective divisor.

One way to characterize the infinitude of $\#C_d'(K)$ is the following theorem.

\begin{thm}[{\cite[Theorem 4.2. (1)]{BELOV}}]\label{translate_of_abelian_variety_thm}
    Let $C$ be a curve over a number field $K$. Then $\#C_d'(K) = \infty$ if and only if at least one of the following two statements holds:
    \begin{enumerate}[(1)]
        \item There exists a map $C\to \mathbb{P}^1$ of degree $d$ over $K$.
        \item There exists a degree $d$ point $x\in C$ and a positive rank abelian subvariety $A\subset \Pic^0 (C)$ such that $\phi_d(x_1+\ldots+x_d)+A\subset W^0_dC$, where $x_1,\ldots,x_d$ are the Galois conjugates of $x$.
    \end{enumerate} 
\end{thm}

Regarding the minimum density degree, Kadets and Vogt \cite{KV2025} gave a characterization when $\min(\delta(C/K))$ and $\min(\wp(C/K))$ are equal to $d$ for small values of $d$, which encompasses the previous results of Harris-Silverman \cite{HarrisSilverman91} and Abramovich-Harris \cite{AbramovichHarris91}.

\begin{thm}[{\cite[Theorem 1.2]{KV2025}}]\label{kadetsvogt1.2}
    Suppose that $C$ is a nice curve over a number field $K$. Then the following statements hold:
    \begin{enumerate}[(1)]
        \item If $\min(\delta(C/K))=2$, then $C$ is a double cover of $\mathbb{P}^1$ or an elliptic curve of positive rank over $K$.
        \item If $\min(\delta(C/K))=3$, then one of the following three cases holds:
        \begin{enumerate}[(a)]
            \item $C$ is a triple cover of $\mathbb{P}^1$ or an elliptic curve of positive rank over $K$.
            \item $C$ is a smooth plane quartic with no rational points, positive rank Jacobian, and at least one cubic point.
            \item $C$ is a genus $4$ Debarre-Fahlaoui curve.
        \end{enumerate}
        \item If $\min(\wp(C/K))=d\leq3$, then $C_{\overline{K}}$ is a degree $d$ cover of $\mathbb{P}^1$ or an elliptic curve.
        \item If $\min(\wp(C/K))=d=4,5$, then either $C_{\overline{K}}$ is a Debarre-Fahlaoui curve, or $C_{\overline{K}}$ is a degree $d$ cover of $\mathbb{P}^1$ or an elliptic curve.
    \end{enumerate}
\end{thm}

Debarre-Fahlaoui curves mentioned here are defined in \Cref{debarrefahlaoui_def}.

\begin{definition}
    For a curve $C$ defined over a field $K$, the $K$-gonality $\textup{gon}_K C$ is the smallest integer $d$ such that there exists a morphism of degree $d$ from $C$ to $\mathbb{P}^1$ defined over $K$.
\end{definition}

The question of determining whether the curve $C/K$ has infinitely many points of degree $d$ over $K$ is closely related to the $K$-gonality of $C$. For example, Frey \cite{frey} proved that if a curve $C$ defined over a number field $K$ has infinitely many points of degree $\leq d$ over $K$, then $\textup{gon}_K C\leq2d$.

We know from \cite[Theorem 1.1]{NajmanOrlic22} all curves $X_0(N)$ with $\Q$-gonality at most $6$. In particular, the curve $X_0(N)$ has $\Q$-gonality equal to $5$ if and only if $N=109$. Now we want to get an upper bound for $N$ such that $X_0(N)$ admits a degree $5$ morphism to an elliptic curve. We start with the following result that gives a linear lower bound on the $\C$-gonality of the modular curve in terms of the index of its congruence subgroup.

\begin{thm}[{\cite[Section 0.2]{abramovich} and Appendix 2 to \cite{Kim2002}}]\label{abramovichbound}
    Let $X_\Gamma$ be the algebraic curve corresponding to a congruence subgroup $\Gamma\subseteq \SL_2(\Z)$ of index
    $D_\Gamma=[\SL_2(\Z):\pm\Gamma]$. Then \[D_\Gamma\leq \frac{2^{15}}{325}\textup{gon}_\C(X_\Gamma).\]
\end{thm}

\begin{proof}
    Abramovich \cite{abramovich} proved the inequality \[\lambda_1 D_\Gamma\leq 24\cdot\textup{gon}_\C(X_\Gamma), \] where $\lambda_1$ is the smallest positive eigenvalue of the Laplacian operator on the Hilbert space $L^2(X_\Gamma)$. Kim and Sarnak \cite{Kim2002} showed that $\lambda_1\geq\displaystyle\frac{975}{4096}$.
\end{proof}

\begin{remark}
    In the same paper \cite{abramovich}, Abramovich proved that $\lambda_1\geq\displaystyle\frac{21}{100}$ (this was later improved to $\displaystyle\frac{975}{4096}$ by Kim and Sarnak). Selberg's eigenvalue conjecture \cite{selberg} states that $\lambda_1\geq\displaystyle\frac{1}{4}$.
\end{remark}

We can also get a bound on the $\Q$-gonality of the curve $X_0(N)$ by counting the $\F_{p^2}$-points.

\begin{lem}[Ogg, \cite{Ogg74}]\label{Ogg} For a prime $p$ with $p\nmid N$, let $\#X_0(N)(\F_{p^2})$ denote the number of $\F_{p^2}$-rational points on $X_0(N)$.
Then $$\#X_0(N)(\F_{p^2})\geq L_p(N):=\frac{p-1}{12} \psi(N)+2^{\omega(N)},$$
where $\psi(N)=N\displaystyle\prod_{{q|N}\atop{q \textup{ prime}}}\left(1+\frac{1}{q}\right)$ and $\omega(N)$ is the number of distinct prime factors of $N$.
\end{lem}

Lemma~\ref{Ogg} is not explicitly proven in \cite{Ogg74}, so we briefly provide 
the underlying point-counting argument.
Let $p\nmid N$ be a prime. Then $X_0(N)$ has good reduction at $p$.
Over $\mathbb{F}_{p^2}$, $X_0(N)(\mathbb{F}_{p^2})$ contains at least $2^{\omega(N)}$ cusps.
For non-cuspidal points, one considers supersingular elliptic curves in characteristic $p$.
For a supersingular elliptic curve $E$, the Frobenius endomorphism acts as multiplication by an integer over
$\mathbb{F}_{p^2}$, which implies that all $\psi(N)$ cyclic subgroups $C$ of $E$ of order $N$ are defined over $\mathbb{F}_{p^2}$.
Each such pair $(E,C)$ yields a non-cuspidal $\mathbb{F}_{p^2}$-rational point on $X_0(N)$.
After identifying points via automorphisms of $E$, one obtains at least
$(p-1)\psi(N)/12$ distinct non-cuspidal $\mathbb{F}_{p^2}$-rational points
lying above supersingular points of $X(1)$.
This follows from the Deuring-Eichler mass formula; see, for instance,
\cite[Lemma~3.20]{BGGP}.

\begin{cor}\label{cor:ogg}Let $N$ be an integer and $p$ a prime not dividing $N$, and assume that there exists a map $f : X_0(N) \to E$ of degree $d$ to an elliptic curve $E$ over $\Q$. Then
\begin{align*}
    d &\geq \frac {1} {\#E(\F_{p^2})}\left(\frac{p-1}{12} \psi(N)+2^{\omega(N)}\right) \geq \frac {1} {(p+1)^2}\left(\frac{p-1}{12} \psi(N)+2^{\omega(N)}\right).
\end{align*}
\end{cor}
\begin{proof}Since both $X_0(N)$ and $E$ have good reduction at $p$, $f$ induces a map
$f_{\F_p} \colon  X_0(N)_{\F_p}\to  E_{\F_p}$
of degree $d$ (cf. \cite[Proof of Theorem 3.2]{NguyenSaito}). The first inequality now follows from \Cref{Ogg}, since every point in $E(\F_{p^2})$ has at most $d$ points of $X_0(N)(\F_{p^2})$ mapping to it under $f_{\F_p}$.
The second inequality is implied by the Hasse bound $\# E(\F_{p^2})\leq (p+1)^2$  \cite[Theorem V.1.1]{silverman}.

\end{proof}

\section{Positive rank pentaelliptic curves}\label{positive_rank_pentaelliptic}


Note that if a curve $C$ admits a $\Q$-rational map of degree $5$ onto $\mathbb P^1$ or an elliptic curve of positive rank, then one can easily prove that $C_5(\Q)$ is infinite; hence we now first consider when a curve admits a $\Q$-rational map of degree 5 to a positive rank elliptic curve.

\begin{definition}
If a curve $C$ over a number field $K$ of genus $g(C)\geq 2$ admits a $K$-rational map of degree 5 to an elliptic curve $E$, then we say that $C$ is {\it pentaelliptic} over $K$.
If $C$ admits a map of degree $5$ to an elliptic curve $E$ such that $E(K)$ has Mordell-Weil rank at least $1$, then we call $C$ {\it positive rank pentaelliptic} over $K$.
\end{definition}

The goal of this section is to prove the following theorem:

\begin{thm}\label{thm:pos_rank_penta}
There are no positive rank pentaelliptic modular curves $X_0(N)$ over $\Q$. 
\end{thm}

The proof of this theorem is split up in two parts. First, in \Cref{lem:pentaelliptc_bound} we will get a lower bound on $N$ for which $X_0(N)$ is not pentaelliptic using Abramovich lower bound on the $\C$-gonality (\Cref{abramovichbound}) and Ogg's method (\Cref{cor:ogg}). After this result, there are finitely many $N$ left and they are dealt with in \Cref{deg5toellcurve_prop}, based on the techniques from \Cref{coefficient_divisible_by_degf}.

With \Cref{thm:pos_rank_penta} and \Cref{new_cond_quintic} in hand, it is now easy to limit the values of $N$ for which $\min(\delta(X_0(N)/\Q)) = 5$ to a reasonably small set.

\begin{prop}\label{new_cond_quintic_cor1}

If  $\min(\delta(X_0(N)/\Q)) = 5$, then $N$ has to be one of the following values:
\begin{align*}
N \in \{& 76, 84, 90, 93, 97, 108, 109, 112, 113, 115, 117, 127, 133, 137, 139, 147, 169 \}.
\end{align*}
\end{prop}

\begin{proof}
If $N>191$, then $X_0(N)$ has $\Q$-gonality at least $6$ by \cite[Theorem 1.1]{NajmanOrlic22}. Furthermore, if $N>191$ then $X_0(N)$ also has genus $\geq 12$, and by \Cref{thm:pos_rank_penta}, there are no positive rank pentaelliptic curves over $\Q$. So, one can use \Cref{new_cond_quintic} to conclude that if $\min(\delta(X_0(N)/\Q)) = 5$, then $N\leq 191$. 

The list in the proposition was obtained by starting with all values of $N \leq 191$ but excluding the following:
\begin{itemize}
\item[-] The values of $N$  in \Cref{thmquadratic,thmcubic,quarticthm}, since these have $\min(\delta(X_0(N)/\Q)) \leq 4$.
\item[-] Those where $\gon_{\Q} X_0(N) \geq 6$ and $g(X_0(N)) \geq 12$, since there one can use \Cref{new_cond_quintic}.
\end{itemize}

\end{proof}

After \Cref{new_cond_quintic_cor1}, we are left with only 17 levels $N$ which we need to check in order to prove \Cref{densitydeg5thm}. But we will first give the proof of \Cref{thm:pos_rank_penta} which is still missing.

\subsection{Proof of \texorpdfstring{\Cref{thm:pos_rank_penta}}{the pentaelliptic Theorem}}
\begin{lem}\label{lem:pentaelliptc_bound} If $N \geq 468$, then $X_0(N)$ cannot be positive rank pentaelliptic over $\Q$.
\end{lem}
\begin{proof}
This was proved in \Cref{new_cond_quintic_cor1} because the curve $X_0(N)$ is not even pentaelliptic over $\Q$ for these levels $N$.
\end{proof}

If an elliptic curve $E/\Q$ has conductor $N$, then all possible degrees of rational morphisms from $X_0(N)$ to $E$ are multiples of the modular degree of $E$. Therefore, this case can be easily solved.

However, there were some curves $X_0(N)$ which cannot be eliminated using Ogg's method (\Cref{cor:ogg}) and for which there exists a positive rank elliptic curve $E$ with conductor $\textup{cond}(E)=M\mid N$, $M\neq N$. For these curves, we used the method from \cite{DerickxOrlic23} of determining all possible degrees of rational morphisms from $X_0(N)$ to $E$. More precisely, we used the following results:

\begin{prop}[{\cite[Proposition 1.10]{DerickxOrlic23}}]\label{quadraticformprop}
    Let $C$ be a curve over $\mathbb Q$ with at least one rational point and $E$ an elliptic curve over $\mathbb Q$ that occurs as an isogeny factor of $J(C)$ with multiplicity $n \geq 1$. Then the degree map $\deg: \textup{Hom}(C,E) \to \Z$ can be extended to a positive definite quadratic form on $\textup{Hom}(J(C),E) \cong \Z^n$.
\end{prop}

We will use 
\begin{align*}
\left<\,\_\,,\,\_\,\right> : \textup{Hom}(J(C),E)\times\textup{Hom}(J(C),E) & \to  \textup{End}(J(E)),
\end{align*} 
to denote the quadratic form from \Cref{quadraticformprop}. See \cite[Definition 2.2]{DerickxOrlic23} for more details on how this quadratic form is defined.

\begin{prop}[{Extension of \cite[Proposition 1.11]{DerickxOrlic23}}]\label{quadraticformcoefficients}
    Let $E$ be an elliptic curve of positive $\Q$-rank and conductor $\textup{Cond}(E)=M\mid N<778$, and let $f:X_0(M)\to E$ be a modular parametrization of $E$ of minimal degree. Then (with the natural embedding $X_0(N)\to J_0(N)$), the maps $f\circ \iota_{d,N,M}$ form a basis for $\textup{Hom}_\Q(J_0(N),E)$, where $d$ ranges over all divisors of $\frac{N}{M}$.
\end{prop}

\begin{proof}
    In \cite[Proposition 1.11]{DerickxOrlic23}, this was proved using Sage for $N<408$. The same code works for $N<778$. The outcome of running this code is available in the MD Sage github repository at \begin{center}\url{https://github.com/koffie/mdsage/blob/v0.1.0/articles/derickx_orlic-quartic_X0/computational_proofs.ipynb}
    \end{center}

    The code deals with elliptic curves of odd analytic rank, i.e., those with Atkin-Lehner sign $\epsilon=1$ in the functional equation (\cite[Section 10.5.1]{stein07})\[\Lambda(E,s)=\epsilon\Lambda(E,2-s).\] There are exactly 12 elliptic curves over $\Q$ of conductor $N<778$ with rank at least $2$. Among these, the one with smallest conductor is  \lmfdbec{389}{a}{1}. In order to deal with these 12 cases, one can argue as follows. The only way to have $\textup{Cond}(E) | N < 778$ for one of these cases is when $\textup{Cond}(E)=N$, so $E$ occurs with multiplicity $1$ as an isogeny factor. In particular, injectivity follows from the minimality assumption on the degree of $f:X_0(M)\to E$.
\end{proof}

In this article, we only need \Cref{quadraticformcoefficients} for $N<468$. The reason to stop at $N<778$ is that for $N=778$, the code would need to be modified to also account for the rank 2 elliptic curve \lmfdbec{389}{a}{1}.

In \cite[Theorem 2.13]{DerickxOrlic23}, there is an explicit formula for the pairing $\left<f\circ\iota_{d_1,N,M},f\circ\iota_{d_2,N,M}\right>$ whenever $\frac{N}{M}$ is either squarefree or coprime to $M$. Even if these assumptions are not satisfied, we can prove that the coefficients are divisible by $\deg f$.

\begin{prop}\label{coefficient_divisible_by_degf}
    Let $E$ be an elliptic curve of conductor $M$ and let $f:X_0(M)\to E$ be the modular parametrization of $E$. Assume that $M\mid N$ and let $d_1$ and $d_2$ be divisors of $\frac{N}{M}$. Then the integer coefficient
    $$\left<f\circ\iota_{d_1,N,M},f\circ\iota_{d_2,N,M}\right>$$ is divisible by $\deg f$.
\end{prop}

\begin{proof}
    We can prove that $$\left<f\circ\iota_{d_1,N,M},f\circ\iota_{d_2,N,M}\right>=[\deg f]\circ\left<\iota_{d_1,N,M},\iota_{d_2,N,M}\right>,$$ see \cite[Start of the proof of Theorem 2.13]{DerickxOrlic23} for more details.
\end{proof}

The following result collects some data that is already available in the LMFDB \cite{lmfdb} and is used in the proof of \Cref{deg5toellcurve_prop}.
\begin{lem}\label{modulardeg5prop}
    The elliptic curves with modular degree $5$ are exactly those with LMFDB label \lmfdbec{11}{a}{1}, \lmfdbec{11}{a}{3}, \lmfdbec{46}{a}{2}, \lmfdbec{67}{a}{1}, and \lmfdbec{89}{b}{2}. All the aforementioned elliptic curves have Mordell-Weil rank 0.
\end{lem}

\begin{proof}
   By \Cref{lem:pentaelliptc_bound}, any elliptic curve of modular degree $5$ has to have conductor at most $1007$. The LMFDB contains all elliptic curves over $\Q$ of conductor at most 500,000. So the full list can be obtained by querying the LMFDB API for all elliptic curves of modular degree 5. The results of this query can be found at 
    \begin{center}\url{https://www.lmfdb.org/api/ec_curvedata?degree=5} \,.\end{center}
\end{proof}

\begin{prop}\label{deg5toellcurve_prop}
    If $N < 468$, then $X_0(N)$ is not positive rank pentaelliptic over $\Q$.
\end{prop}

\begin{proof}
    Suppose that for some $N<468$, the curve $X_0(N)$ admits a degree $5$ rational morphism to a positive rank elliptic curve $E$. Then we must have $\textup{cond}(E)=M\mid N$. By \Cref{quadraticformprop}, all possible degrees of a rational morphism from $X_0(N)$ to $E$ are given by a positive definite quadratic form in $n$ variables, where $n$ is the number of positive divisors of $\frac{N}{M}$.

    By \Cref{quadraticformcoefficients}, the coefficients of this quadratic form are given by $$\left<f\circ\iota_{d_1,N,M},f\circ\iota_{d_2,N,M}\right>$$ and they are all divisible by $\deg f$ by \Cref{coefficient_divisible_by_degf}. Therefore, in order for this quadratic form to attain the value $5$, we must have $\deg f\mid 5$.

    If $\deg f=1$, then $X_0(M)=E$ is an elliptic curve. However, all curves $X_0(M)$ with genus $1$ have $\Q$-rank $0$ \cite[p. 449]{Ogg74}. Therefore, this case is impossible.

    If $\deg f=5$, then \Cref{modulardeg5prop} tells us that there are exactly $5$ such elliptic curves, and all of them have $\Q$-rank $0$, giving us a contradiction in this case as well.
\end{proof}

\section{Modular curves of minimum density degree 5}\label{AVtranslates_section}
This is the last section that contains results needed to prove \Cref{densitydeg5thm}. 
We use several different techniques in order to deal with the $17$ cases that still needed consideration after \Cref{new_cond_quintic_cor1}.

We begin by outlining the overall computational strategy used in this section,
and by indicating which technique provides the decisive input for each remaining
level~$N$.

\begin{itemize}
  \item \emph{Jacobian rank computations.}
  For levels~$N$ such that the Jacobian $J_0(N)$ has rank~$0$ over~$\Q$,
  finiteness of degree~$5$ points follows directly
  (Proposition~\ref{jacobian_rank0_prop}).

  \item \emph{Dimension bounds for $W^0_5$.}
  When $J_0(N)$ has positive rank, bounds on the dimension of abelian
  subvarieties of $W^0_5(X_0(N))$ are used to exclude the existence of
  positive-rank translates
  (Corollary~\ref{no_AV_of_dim_12}).

  \item \emph{Debarre--Fahlaoui curves.}
  In certain cases, the only remaining theoretical possibility would be that
  $X_0(N)$ arises from a Debarre--Fahlaoui curve; this possibility is ruled out
  using the structure of such curves
  (Proposition~\ref{DFcurves}).

  \item \emph{Translates of abelian varieties and specialization.}
  For the exceptional levels $N=112$ and $N=117$, we combine specialization results
  for translates of abelian varieties with explicit computations over finite
  fields to exclude the existence of positive-rank translates in $W^0_5$
  (Proposition~\ref{prop112117}).
\end{itemize}

We note that the first 3 methods only involve gathering facts on the isogeny decomposition of $J_0(N)$ from the LMFDB combined with a theoretical argument. The last method, Translates of abelian varieties and specialization, is different in nature and involves extensive computer computations.

For the reader's convenience, Table~\ref{tab:section4-summary} summarizes the
technique that provides the decisive argument for each level~$N$ appearing in
Section~\ref{AVtranslates_section}.

\begin{table}[h]
\centering
\begin{tabular}{|l|l|}
\hline
\multicolumn{1}{|c|}{Decisive method} & 
\multicolumn{1}{c|}{Level $N$} \\ \hline

Jacobian rank computation 
& $76,84,90,105,108,110,120,126,$ \\
& $132,140,144,150,168,180$ \\ \hline

Dimension bounds for $W^0_5$
& $97,113,127,137,139,169$ \\ \hline

Debarre--Fahlaoui curves
& $93,115,133,147$ \\ \hline

Translates of abelian varieties
& $112,117$ \\ 
and specialization & \\ \hline 
\end{tabular}
\bigskip
\caption{Decisive computational methods used in Section~4.}
\label{tab:section4-summary}
\end{table}

\subsection{Jacobian rank and dimension bounds}\label{subsec:jacobian-dimension}

We first treat those levels $N$ for which the finiteness of degree~$5$ points can
be established using information on the Jacobian $J_0(N)$.
This includes the cases where $J_0(N)$ has rank~$0$ over~$\Q$, as well as the cases
where the dimension bounds on abelian subvarieties of $W^0_5(X_0(N))$ exclude the
existence of positive-rank translates.

\begin{prop}\label{jacobian_rank0_prop}
    The modular curve $X_0(N)$ has only finitely many degree $5$ points for \[N\in\{76,84,90,105,108,110,120,126,132,140,144,150,168,180\}.\]
\end{prop}

\begin{proof}
    In all these cases, the Jacobian $J_0(N)$ has rank $0$ over $\Q$ by \cite[Theorem 3.1 (1)]{Deg3Class}. Therefore, case $(2)$ of \Cref{translate_of_abelian_variety_thm} cannot happen. Hence, it suffices to prove that there are no degree $5$ rational morphisms $X_0(N)\to\PP^1$. Indeed, these curves $X_0(N)$ have $\Q$-gonality at least $6$ by \cite[Theorem 1.1]{NajmanOrlic22} and we are done.
\end{proof}

\begin{prop}[{\cite[Proposition 3.6]{DebarreFahlaoui}}]\label{DF_AV_dim}
    Let $C$ be a curve over $\C$ of genus $g$. Let $W_d^r(C)$ be as in \Cref{eq:Wdr_over_C}. Assume that $W_d^r(C)$ contains an abelian variety $A$ and that $d\leq g-1+r$. Then $\dim A\leq d/2-r$.
\end{prop}

\begin{cor}\label{no_AV_of_dim_12}
    The modular curve $X_0(N)$ has only finitely many degree $5$ points for \[N\in\{97,113,127,137,139,169\}.\]
\end{cor}

\begin{proof}
    In all these cases, we have $g(X_0(N))\geq6$ and the $\Q$-gonality of $X_0(N)$ is at least $6$ by \cite[Theorem 1.1]{NajmanOrlic22}. \Cref{translate_of_abelian_variety_thm} now tells us that if the curve $X_0(N)$ has infinitely many degree $5$ points, then $W_5^0(C)$ must contain a translate of a positive rank abelian variety $A$. By \Cref{DF_AV_dim}, we must have $\dim A\leq2$.

    However, in all these cases, all simple abelian subvarieties $A$ of dimension $\leq2$ in the decomposition of the Jacobian $J_0(N)$ have analytic rank $0$. This can be checked on LMFDB by searching all newforms of level dividing $N$, weight $2$ and character order $1$. An example for $N=97$ is \begin{center}
    \url{https://www.lmfdb.org/ModularForm/GL2/Q/holomorphic/?level_type=divides&level=97&weight=2&char_order=1&showcol=analytic_rank}.
    \end{center} Therefore, their rank is also $0$ and we get a contradiction.
\end{proof}

\subsection{Debarre-Fahlaoui curves}
In \cite[Proposition 6.3]{KV2025}, there is a partial classification of the curves $C$ with $\min(\delta(C/K))=5$. Before we state it, we need to give the definition of Debarre-Fahlaoui curves occurring in that proposition. See \cite[Section 5]{KV2025} for more details.

Let $E$ be an elliptic curve over a number field $K$. Consider the map \begin{align}
    \begin{split}\label{eq:sym2_of_EC}
\pi : \Sym^2 (E) &\to E \\
    x +_1 y  &\mapsto x+_2 y
        \end{split}
\end{align}
where $+_1$ means addition of divisors on $E$ and $+_2$ is the group law of the elliptic curve.

On $\Sym^2 (E)$, there are two ways to describe divisors. For a point $x$ on $E$, we let $F_x := \pi^{-1}(x)$. All the $F_x$ are numerically equivalent and we denote by $F$ this equivalence class. Similarly, for a point $x$ on $E$, we can also define the divisor $H_x$ to be the subvariety of $\Sym^2 (E)$ consisting of all degree $2$ divisors containing $x$. Again, all the $H_x$ are numerically equivalent, and we denote by $H$ the numerical equivalence class.

\begin{definition}\cite[Definition 5.1]{KV2025}\label{debarrefahlaoui_def} Let $E$ be a positive rank elliptic curve over a number field $K$. A \textit{Debarre–Fahlaoui} curve is a geometrically integral curve on $\Sym^2 (E)$ in the numerical class $(d +m)H -mF$, for some integers $d,m$ with $1 \leq m \leq d$.
\end{definition}

\begin{prop}[{\cite[Proposition 6.3]{KV2025}}]\label{kadetsvogt_prop_DF}
    Suppose that $C$ is a nice curve of genus $g$ over a number field $K$ such that $\min(\delta(C/K))=5$. Then one of the following cases holds:
    \begin{enumerate}[(1)]
        \item There is a degree $5$ $K$-rational morphism from $C$ to $\PP^1$ or a positive $K$-rank elliptic curve.
        \item $C$ is the normalization of a Debarre-Fahlaoui curve.
        \item $g=5,6,7,8$.
    \end{enumerate}
\end{prop}

\begin{proposition}\label{prop:BFtoEC} Let $E$ be an elliptic curve and $C \subseteq \Sym^2 (E)$ be a geometrically integral curve of genus $>0$. Then $C$ admits a non-constant map to $E$.
\end{proposition}
\begin{proof}
Consider the map $\pi : \Sym^2 (E) \to E$ from \cref{eq:sym2_of_EC}. The fibers of $\pi$ are all isomorphic to $\PP^1$. In particular, $C$ cannot be contained in a fiber of $\pi$. This means that $\pi$ has to induce a non-constant map $\pi|_C : C \to E$.
\end{proof}
\begin{prop}\label{DFcurves}
    The modular curve $X_0(N)$ has only finitely many degree $5$ points for \[N\in\{93,115,133,147\}.\]
\end{prop}

\begin{proof}
    In these cases, all the simple factors of the Jacobian $J_0(N)$ of positive rank have dimension $2$. Therefore, we cannot use the same argument as in \Cref{no_AV_of_dim_12}.

    For the sake of contradiction, we assume that $X_0(N)$ has infinitely many degree $5$ points. All these curves have genus $g \geq 9$. Also, the $\Q$-gonality of these curves is at least $6$ by \cite[Theorem 1.1]{NajmanOrlic22} and there are no positive rank elliptic curves of conductor dividing $N$. Therefore, by \Cref{kadetsvogt_prop_DF}, the curve $X_0(N)$ must be the normalization of a Debarre-Fahlaoui curve.

    However, applying \Cref{prop:BFtoEC} to the definition of a Debarre-Fahlaoui curve tells us that in that case, we must have a non-constant morphism from $X_0(N)$ to a positive rank elliptic curve, which contradicts the aforementioned fact that there are no positive rank elliptic curves of conductor dividing $N$ for these levels $N$.
\end{proof}

\subsection{Translates of abelian varieties and specialization} 

The only levels $N$ that still need to be checked to prove \Cref{densitydeg5thm} are $N=112$ and $N=117$. We were not able to solve these two cases using the previous methods.

Before handling the remaining cases, we need some results on what happens with translates of abelian varieties under specialization. In this section we will use the definition of $W_d^0 C$ as in \Cref{eq:Wd0}.

\begin{lem} \label{lem:reduction_AV_translate} Let $R$ be a discrete valuation ring with fraction field $K$ and residue field $k$. Let $C$ be a smooth projective curve over $R$ with geometrically irreducible fibers such that $C(R)\neq \emptyset$. Let $\mathcal A$ be an abelian subvariety of $\Pic^0(C)_K$ and let $A \subset \Pic^0(C)$ be its Zariski closure. Suppose that $d$ is an integer such that $W_d^0 C_{K} \subseteq \Pic^d(C)_{K}$ contains a translate of $\mathcal A = A_K$. Then $W_d^0 C_{k}$ contains a translate of $A_{k}$.
\end{lem}
\begin{proof}
Notice that $\Sym^d(C)$ is proper over $R$; this can be seen from the valuative criterion of properness. Indeed, points on $\Sym^d(C)$ can be interpreted as horizontal effective divisors of degree $d$, and hence the valuative criterion of properness is satisfied because divisors on a curve over the fraction field of a discrete valuation ring can uniquely be extended to a horizontal divisor over the discrete valuation ring itself.

Since $\Sym^d(C)$ is proper, it is universally closed. In particular $W_d^0 C$, being the image of $\Sym^d(C)$ in $\Pic^0(C)$, is a closed subset. By assumption $W_d^0 C$ contains $\mathcal A$, but since $W_d^0 C$ is closed, it will also contain its closure $A$. The desired statement now follows by taking the special fiber over $k$ of the inclusion $A \subseteq W_d^0 C$.
\end{proof}

\begin{prop}\label{prop112117}
If $C=X_0(112)$ or $C=X_0(117)$, then $W_5^0C_\Q$ does not contain a translate of a positive rank abelian variety.
\end{prop}
\begin{proof}
According to the LMDFB, each of $J_0(112)$ and $J_0(117)$ contains exactly one simple abelian variety of positive analytic rank. This can be seen on
\begin{center}
    \url{https://www.lmfdb.org/ModularForm/GL2/Q/holomorphic/?level_type=divides&level=112&weight=2&char_order=1&analytic_rank=1-}
\end{center} and \begin{center}
    \url{https://www.lmfdb.org/ModularForm/GL2/Q/holomorphic/?level_type=divides&level=117&weight=2&char_order=1&analytic_rank=1-}\,.
\end{center}

The abelian varieties attached to these modular forms are the elliptic curves \lmfdbec{112}{a}{2} and \lmfdbec{117}{a}{4}, respectively, both of which have positive algebraic rank as well. Let $E$ be the elliptic curve  \lmfdbec{112}{a}{2} when $C=X_0(112)$, and $E$ be the elliptic curve \lmfdbec{117}{a}{4} when $C=X_0(117)$. Then, in order to prove the proposition, it suffices to show that $W_5^0 C_\Q$ does not contain a translate of $E$. We will do this using \Cref{lem:reduction_AV_translate} and some explicit computations.

We computed explicit models for $X_0(112)$ and $X_0(117)$ together with maps to \lmfdbec{112}{a}{2} respectively \lmfdbec{117}{a}{4}, see \Cref{sec:computing_models} for more details of how we computed this.

Let $p$ be a prime of good reduction for $C$. Note that since $C$ contains the rational cusp $\infty$, we know that $\Sym^5C(\F_p) \to W_5^0 C_{\F_p}$ is surjective. This is because each fiber is isomorphic to $\PP^n_{\F_p}$ for some integer $n\geq 0$.

For $C=X_0(112)$, we work over $\F_3$ and we will show that $W_5^0 C_{\F_3}$ does not contain a translate of $E_{\F_3}$. Using Magma we enumerated the finitely many points on $\Sym^5C(\F_3)$. If we manage to show that $D+E_{\F_3} \subsetneq W_5^0 C_{\F_3}$ for each $D \in \Sym^5C(\F_3)$, then we are done. In order to show that $D+E_{\F_3} \subsetneq W_5^0 C_{\F_3}$, we proceeded as follows.

Note that $E(\F_3)$ is a cyclic group of order $6$. We computed a generator $P$ of $E(\F_3)$. For each $D \in \Sym^5C(\F_3)$ we found an integer $i_D \leq 5$ such that $D+i_DP \notin W_5^0 C(\F_3)$, which shows that $D+E_{\F_3} \subsetneq W_5^0 C_{\F_3}$ and we are done.

For $C=X_0(117)$ we performed a similar computation, however this time working over $\F_5$. In this case $E(\F_5) \cong \Z/2\Z \times \Z/4\Z$. This time we used a point $P \in E(\F_5)$ and showed that for each $D \in \Sym^5C(\F_5)$ there is an integer $i_D \leq 3$ such that $D+i_DP \notin W_5^0 C(\F_5)$.

The code for these computations is in the file \github{X0_112.m} and \github{X0_117.m}.

\end{proof}

\begin{cor}\label{cor_112_117}
    The modular curve $X_0(N)$ has only finitely many degree $5$ points for \[N\in\{112,117\}.\]
\end{cor}

\begin{proof}
    These two curves have $\Q$-gonality equal to $6$ by \cite[Theorem 1.4]{NajmanOrlic22}. Also, by \Cref{prop112117}, if $C=X_0(112)$ or $C=X_0(117)$, then $W_5^0C_\Q$ does not contain a translate of a positive rank abelian variety. Therefore, by \Cref{translate_of_abelian_variety_thm}, these two curves have only finitely many degree $5$ points.
\end{proof}

\subsection{Explicitly computing modular parametrizations }\label{sec:computing_models}
Let $E$ be an elliptic curve over $\Q$ of conductor $N$. Then $E$ admits a modular parametrization, i.e. there is a map $\phi: X_0(N) \to E$. 
In this section we will shortly explain how we obtained an explicit algebraic equation for $X_0(N)$ and the corresponding map $X_0(N) \to E$. 

Let $\{f_1(\tau),\dots, f_g(\tau)\}$ be a basis for the space of weight 2 cuspforms, where $g$ is the genus of the curve $X_0(N)$.
Defining $\omega_i = f_i(\tau) d\tau$, the set $\{\omega_i\}$ then forms a basis for the space of holomorphic differentials on $X_0(N)$.
If $X_0(N)$ is not hyperelliptic, the map $\phi: X_0(N) \to \mathbb{P}^{g-1}$ defined by \[\phi(P)=[\omega_1(P):\cdots:\omega_g(P)]=[f_1(P):\cdots:f_g(P)]\] 
gives the canonical embedding.
In this case, we choose three appropriate cuspforms $f_i, f_j, f_k$ so that defining $X = f_i / f_k$ and $Y = f_j / f_k$ ensures that $X$ and $Y$ generate the function field of $X_0(N)$. This can always be done when $X_0(N)$ is not hyperelliptic since then $X_0(N)$ is isomorphic to its image in $\PP^{g-1}$, and from there one can find $f_i,f_j,f_k$ by choosing a suitable projection to $\PP^2$ as in \cite[Exercise I.4.9.]{hartshorne1977algebraic} to get a plane model. We do not know beforehand which of the one forms will be a suitable choice for $f_i,f_j,f_k$. However, in practice most choices will work, so we just choose 3 of them and then later verify that the model of $X_0(N)$ obtained has the correct geometric genus. This verification suffices, since having the same geometric genus forces $X_0(N)$ to be birational to its image in $\PP^2$.
The $q$-expansions of $X$ and $Y$ can be computed from the $q$-expansions of $f_i, f_j$, and $f_k$.
Using \cite[Algorithm 4.2]{Jeon2022}, we obtain an equation $F_N(X,Y)$ satisfied by $X$ and $Y$, which gives a plane equation for $X_0(N)$.

For example, consider the curve $X_0(112)$. 
Let us choose the following three weight 2 cuspforms on $X_0(N)$:
\begin{align*}
f_1=& q^9 - 2q^{13} + 4q^{17} - q^{21} - 3q^{25} - 4q^{29} + \cdots \\
f_2=& 2q^{10} - q^{14} - 2q^{18} - 2q^{22} + 2q^{26} + \cdots \\
f_3=& q^{11} + q^{15} - 2q^{19} - q^{23} - \cdots
\end{align*}
Now, defining 
\begin{align}\label{xy1}
X=&\frac{f_2}{f_1}= 2q + 3q^5 - 4q^9 - 20q^{13} -13q^{17} + \cdots \\  \notag
Y=&\frac{f_3}{f_1}= q^2 + 3q^6 - 12q^{14} - 18q^{18} + \cdots,
\end{align} 
we obtain the following relation satisfied by $X$ and $Y$ using \cite[Algorithm 4.2]{Jeon2022},
\begin{align*}
F_{112}(X,Y)=&X^{10}Y + 25X^6Y^5 + 162X^2Y^9 - 5X^8Y^2 - 77X^4Y^6 - 81Y^{10} \\
&+ X^6Y^3 - 78X^2Y^7 + X^8 + 31X^4Y^4 + 126Y^8 - 4X^6Y - 53X^2Y^5 \\
&- 6X^4Y^2 - 85Y^6 + 17X^2Y^3 + 28Y^4 + X^2Y - 4Y^2=0.
\end{align*}
Since the genus of the curve defined by this equation is 11, which is the same as the genus of $X_0(112)$, 
it follows that $X$ and $Y$ generate the function field of $X_0(112)$, and thus, $F_{112}(X,Y)=0$ is a defining equation of $X_0(112).$

In order to explicitly get a modular parametrization, we need two functions $F_1,F_2 \in \Q(X_0(N))$ such that $F_1,F_2$ are the $x$ and $y$ coordinates of a map to a long Weierstrass form of the elliptic curve $E$. The function \texttt{elltaniyama} from PARI/GP \cite[\S 3.15.83]{PARIGP} takes a long Weierstrass form of the elliptic curve $E$ as input, and from this it can determine the $q$-expansions of such $F_1$ and $F_2$ up to arbitrary precision. After having obtained these $q$-expansions, we can use \cite[Algorithm 4.3]{Jeon2022} to express $F_1$ and $F_2$ in terms of $X$ and $Y$.

Let us reconsider the case of $X_0(112)$. 
First, the $q$-expansions that satisfy the $x$-coordinate and $y$-coordinate of the long Weierstrass equation of \lmfdbec{112}{a}{2} are as follows.
\begin{align}\label{xy2}
x=&q^{-2} + 1 + 3q^2 + 4q^4 + 9q^6 + 12q^8 + 15q^{10} + \cdots,\\ \notag
y=&-q^{-3} - 2q^{-1} - 5q - 11q^3 - 16q^5 - 31q^7 - 57q^9 - \cdots.
\end{align}
The file \github{modular_parameterisations.gp} contains the PARI/GP code that computes the above q-expansions. Applying \cite[Algorithm 4.3]{Jeon2022} to the $q$-expansions in \eqref{xy1} and \eqref{xy2}, we obtain the following map $(F_1,F_2):X_0(112)\to$\lmfdbec{112}{a}{2} as follows:
{\tiny \begin{align*}
F_1(X,Y)=&-(5X^6Y + 16X^4Y^3 + 72X^2Y^5 + 324Y^7 + 4X^6 - 16X^4Y^2 - 28X^2Y^4 - 36Y^6 - 28X^4Y + 14X^2Y^3 \\
&- 60Y^5 + 5X^4 + 56X^2Y^2 + 38Y^4 - 11X^2Y - 28Y^3 - 8Y^2 + 1)/(5X^6Y - 6X^4Y^3 + 54X^2Y^5 - 25X^4Y^2 \\
&- 15X^2Y^4 - 189Y^6 - 5X^4Y - 8X^2Y^3 + 90Y^5 + X^4 + 25X^2Y^2 + 129Y^4 - 43Y^3 - X^2 - 15Y^2 + 3Y) \\
F_2(X,Y)=& -(56X^8Y + 171X^6Y^3 + 1229X^4Y^5 + 2322X^2Y^7 + 2268Y^9 + 6X^8 - 100X^6Y^2 - 217X^4Y^4 \\
&+ 1119X^2Y^6 - 513Y^8 - 31X^6Y + 154X^4Y^3 + 628X^2Y^5 - 2349Y^7 + 90X^6 + 326X^4Y^2 - 319X^2Y^4 \\
&+ 858Y^6 - 279X^4Y - 286X^2Y^3 + 1049Y^5 + 45X^4 + 72X^2Y^2 - 444Y^4 - 136X^2Y - 200Y^3 + 6X^2 \\
&+ 81Y^2 + 20Y + 6)/(X(27X^6Y^2 - 8X^4Y^4 + 72X^2Y^6 - 60X^4Y^3 + 136X^2Y^5 - 1170Y^7 - 18X^4Y^2 \\
&- 414X^2Y^4 + 744Y^6 - X^4Y + 230X^2Y^3 + 890Y^5 + 6X^4 + 83X^2Y^2 - 270Y^4 - 20X^2Y - 248Y^3 \\
&- 6X^2 + 27Y))
\end{align*}}

Since the degree of this map is $8$, which matches the degree of the modular parametrization 
$\phi$, we conclude that it indeed gives the modular parametrization $\phi$.

The computation for $X_0(117)$ is done analogously. The Maple code for these computations can be found in \github{modular_parameterisation_112a2.mpl} and \github{modular_parameterisation_117a4.mpl}

\subsection{Proof of \texorpdfstring{\Cref{densitydeg5thm}}{the minimum density degree 5 theorem}}\label{main_results_section}

\begin{proof}
If $\min(\delta(X_0(N/\Q))=5$, then $N$ must be one of the following $17$ values by \Cref{new_cond_quintic_cor1}:

$${76, 84, 90, 93, 97, 108, 109, 112, 113, 115, 117, 127, 133, 137, 139, 147, 169}.$$
We will consider these $17$ values case by case.

\begin{itemize}
    \item $N=109$: Because of \cite[Theorem 1.3]{NajmanOrlic22} the curve $X_0(109)_\Q$ admits a morphism of degree $5$ to $\PP^1$ so it has infinitely many points of degree $\leq 5$.
It has only finitely many points of degree $\leq4$ by \Cref{thmquadratic,thmcubic,quarticthm}, hence $\min(\delta(X_0(109)/\Q))=5$.
    \item $N=76, 84, 90, 108$: Here $\min(\delta(X_0(N)/\Q))\neq 5$ due to \Cref{jacobian_rank0_prop}.
    \item $N=97, 113, 127, 137, 139, 169$: Here $\min(\delta(X_0(N)/\Q))\neq 5$ due to \Cref{no_AV_of_dim_12}.
    \item $N=93, 115, 133, 147$: Here $\min(\delta(X_0(N)/\Q))\neq 5$ due to \Cref{DFcurves}.
    \item $N=112,117$: Here $\min(\delta(X_0(N)/\Q))\neq 5$ due to \Cref{cor_112_117}.
\end{itemize}
\end{proof}

After having proved \Cref{densitydeg5thm}, we now proceed to the proof of \Cref{quinticcor}. The cases not covered by \Cref{densitydeg5thm} are those with infinitely many points of degree $\leq4$, and they are listed in \Cref{quarticthm}. 
The remainder of the paper presents results used to establish \Cref{quinticcor}.

\section{Castelnuovo-Severi inequality}\label{cs_section}

We used the Castelnuovo-Severi inequality to prove that some curves $X_0(N)$ do not admit a degree $5$ morphism to $\PP^1$ or an elliptic curve. The reader can look at \cite[Theorem 3.11.3]{Stichtenoth09} for a proof of this inequality. Some of these curves have $\Q$-gonality at most $4$ (as can be seen in \Cref{csprop}), hence we cannot use methods like the $\F_p$-gonality (see \cite[Section 3]{NajmanOrlic22}) to prove that there are no degree $5$ morphisms.

\begin{prop}[Castelnuovo-Severi inequality]\label{tm:CS}
Let $K$ be a perfect field, and let $X,\ Y, \ Z$ be curves over $K$. Let non-constant morphisms $\pi_Y:X\rightarrow Y$ and $\pi_Z:X\rightarrow Z$ over $K$ be given, and let their degrees be $m$ and $n$, respectively. Assume that there is no morphism $X\rightarrow X'$ of degree $>1$ through which both $\pi_Y$ and $\pi_Z$ factor. Then the following inequality holds:
\begin{equation} \label{eq:CS}
g(X)\leq m \cdot g(Y)+n\cdot g(Z) +(m-1)(n-1).
\end{equation}
\end{prop}

Since $\C$ and $\Q$ are both perfect fields due to their characteristic 0, we can apply the Castelnuovo-Severi inequality to get lower bounds on both $\C$ and $\Q$-gonalities.

\begin{prop}\label{csprop}
    The modular curve $X_0(N)$ does not admit a degree $5$ morphism to $\PP^1$ for the following values of $N$. Here $g$ is the genus of the curve $X_0(N)$ and $\deg$ is the degree of the morphism from $X_0(N)$ to $Y$.

\begin{center}
\begin{longtable}{|c|c|c|c|c||c|c|c|c|c|}

\hline
\addtocounter{table}{-1}
$N$ & $g$ & $Y$ & $\deg$ & $g(Y)$ & $N$ & $g$ & $Y$ & $\deg$ & $g(Y)$\\
    \hline

    $46$ & $5$ & $X_0(46)/\left<w_{23}\right>$ & $2$ & $0$ & $59$ & $5$ & $X_0(59)/\left<w_{59}\right>$ & $2$ & $0$\\
    $60$ & $7$ & $X_0(60)/\left<w_{15}\right>$ & $2$ & $1$ & $62$ & $7$ & $X_0(62)/\left<w_{31}\right>$ & $2$ & $1$\\
    $66$ & $9$ & $X_0(66)/\left<w_{11}\right>$ & $2$ & $2$ & $69$ & $7$ & $X_0(69)/\left<w_{23}\right>$ & $2$ & $1$\\
    $70$ & $9$ & $X_0(70)/\left<w_{35}\right>$ & $2$ & $2$ & $71$ & $6$ & $X_0(71)/\left<w_{71}\right>$ & $2$ & $0$\\
    $78$ & $11$ & $X_0(78)/\left<w_{39}\right>$ & $2$ & $2$ & $83$ & $7$ & $X_0(83)/\left<w_{83}\right>$ & $2$ & $1$\\
    $87$ & $9$ & $X_0(87)/\left<w_{87}\right>$ & $2$ & $2$ &
    $89$ & $7$ & $X_0(89)/\left<w_{89}\right>$ & $2$ & $1$ \\ $92$ & $11$ & $X_0(92)/\left<w_{23}\right>$ & $2$ & $1$ &
    $94$ & $11$ & $X_0(94)/\left<w_{47}\right>$ & $2$ & $1$ \\ $95$ & $9$ & $X_0(95)/\left<w_{95}\right>$ & $2$ & $1$ &
    $101$ & $8$ & $X_0(101)/\left<w_{101}\right>$ & $2$ & $1$\\
    $104$ & $11$ & $X_0(104)/\left<w_{104}\right>$ & $2$ & $3$ &
    $107$ & $9$ & $X_0(107)/\left<w_{107}\right>$ & $2$ & $2$\\
    $111$ & $11$ & $X_0(111)/\left<w_{111}\right>$ & $2$ & $2$ &
    $119$ & $11$ & $X_0(119)/\left<w_{119}\right>$ & $2$ & $1$\\
    $131$ & $11$ & $X_0(131)/\left<w_{131}\right>$ & $2$ & $1$ &
    $141$ & $15$ & $X_0(141)/\left<w_{47}\right>$ & $2$ & $3$\\
    $142$ & $17$ & $\PP^1$ & $4$ & $0$ &
    $143$ & $13$ & $\PP^1$ & $4$ & $0$ \\
    $167$ & $14$ & $\PP^1$ & $4$ & $0$ &
    $191$ & $16$ & $\PP^1$ & $4$ & $0$ \\
 
    \hline

\caption*{}
\end{longtable}
\end{center}
\end{prop}

\begin{proof}
    Suppose that there exists a degree $5$ map $f$ from $X_0(N)$ to $\PP^1$. We apply the Castelnuovo-Severi inequality with $f$ and a degree $\deg$ rational morphism $\pi: X_0(N)\to Y$. For $N=142,143,167,191$, we know there is a degree $4$ rational morphism to $\PP^1$ by \cite[Theorem 1.2]{NajmanOrlic22} and for other levels $N$, $\pi$ is a degree $2$ projection map $X_0(N)\to X_0(N)/\left<w_d\right>$.
    
    Since in all cases $\deg=2,4$, we obviously have $\gcd(2,5)=1$ and these two maps clearly cannot both factor through a map of degree $>1$. Therefore, we must have \[g(X_0(N))\leq 5\cdot0+\deg\cdot g(Y)+4\cdot(\deg-1).\]
    
    However, $g(X_0(N))$ is too large in all cases here, and we get a contradiction. This means that there are no degree $5$ maps from these curves $X_0(N)$ to $\PP^1$.
\end{proof}

\begin{prop}\label{cspropellcurve}
    The modular curve $X_0(N)$ does not admit a degree $5$ morphism to an elliptic curve for the following values of $N$. Here $g$ is the genus of the curve $X_0(N)$ and $\deg$ is the degree of the morphism from $X_0(N)$ to $Y$.

\begin{center}
\begin{longtable}{|c|c|c|c|c||c|c|c|c|c|}

\hline
\addtocounter{table}{-1}
$N$ & $g$ & $Y$ & $\deg$ & $g(Y)$ & $N$ & $g$ & $Y$ & $\deg$ & $g(Y)$\\
    \hline

$174$ & $27$ & $X_0(174)/\left<w_{87}\right>$ & $2$ & $8$ & $184$ & $21$ & $X_0(184)/\left<w_{23}\right>$ & $2$ & $5$\\
$222$ & $35$ & $X_0(222)/\left<w_{111}\right>$ & $2$ & $10$ & $231$ & $29$ & $X_0(231)/\left<w_{231}\right>$ & $2$ & $9$\\
$248$ & $29$ & $X_0(248)/\left<w_{31}\right>$ & $2$ & $9$ & $249$ & $27$ & $X_0(249)/\left<w_{83}\right>$ & $2$ & $8$\\
$262$ & $32$ & $X_0(262)/\left<w_{131}\right>$ & $2$ & $9$ & $267$ & $29$ & $X_0(267)/\left<w_{89}\right>$ & $2$ & $9$\\
 
    \hline

\caption*{}
\end{longtable}
\end{center}
\end{prop}

\begin{proof}
    Suppose that there exists a degree $5$ map $f$ from $X_0(N)$ to an elliptic curve. We apply the Castelnuovo-Severi inequality with $f$ and a degree $2$ projection map $\pi: X_0(N)\to Y$. Since $\gcd(2,5)=1$, these two maps clearly cannot both factor through a map of degree $>1$. Therefore, we must have \[g(X_0(N))\leq 5\cdot1+2\cdot g(Y)+4\cdot1.\]
    
    However, $g(X_0(N))$ is too large in all cases here, and we get a contradiction. This means that there are no degree $5$ maps from these curves $X_0(N)$ to an elliptic curve.
\end{proof}

\begin{remark}
    Although in \Cref{quadraticformprop} we have already proved that for these levels $N$ the curve $X_0(N)$ has no morphisms of degree $5$ over $\Q$ to a positive rank elliptic curve (which is what we need to determine the minimum density degree), this result is interesting because it uses a different method and proves that there are no morphisms of degree $5$ over $\C$ to an elliptic curve (with no assumptions on the rank).
\end{remark}

\section{Finding curves with infinitely many degree 5 points}\label{infinitelydeg5section}

The goal of this section is to list as many values of $N$ as possible for which we know that $X_0(N)$ has infinitely many degree $5$ points.

\begin{thm}\label{infinitelymanyquinticthm} If the integer $N$ is one of the following values 
\[N\in\{1-45,47-58,61,63,64,65,67,68,72,73,75,80,81,91,109,121,125\},\]  
then $\Q(X_0(N))$ contains a function of degree $5$ and, in particular, the curve $X_0(N)$ has infinitely many degree $5$ points over $\Q$.
\end{thm}
\begin{proof}For $N=109$ the degree $5$ function exists because of \cite[Theorem 1.3]{NajmanOrlic22}. For the remaining values of $N$, the degree $5$ functions are proved to exist in \Cref{genus2prop,genus34prop,prop125,deg5mapprop}, where we separated the levels in different propositions according to the argument that was used.
The existence of this degree $5$ function together with Hilbert's Irreducibility Theorem \cite[Proposition 3.3.5(2) and Theorem 3.4.1]{serre2016topics} implies the existence of infinitely many quintic points. \end{proof} 

\begin{lem}[{\cite[Lemma 3.1]{derickx2024primitivepoints}}]\label{lem:deg_d_function}
Let $C$ be a nice curve over a number field $K$. Suppose that $C$ has a $K$-rational divisor $D$ of degree $d$ with $d > 2g(C) - 1$. Then $K(C)$ contains a function of degree exactly $d$.
\end{lem}

\begin{prop}\label{genus2prop}
    The modular curve $X_0(N)$ has a function of degree $5$ over $\Q$ for \[N\in\{1-29,31,32,36,37,49,50\}.\]
\end{prop}

\begin{proof}
The divisor $5\infty$ is a divisor of degree $5$ on $X_0(N)_\Q$. Furthermore, all these curves are of genus $g\leq2$. Therefore, $\Q(X_0(N))$ contains a function of degree exactly $5$ for all these values of $N$ by \Cref{lem:deg_d_function}.
\end{proof}

\begin{prop}\label{genus34prop}
    The modular curve $X_0(N)$ has a function of degree $5$ over $\Q$ for \[N\in\{30,33,34,35,38,39,40,41,43,44,45,47,48,53,61,81\}.\]
\end{prop}

\begin{proof}
    These curves are of genus $3$ or $4$, and we have computed that for all of them the infinity cusp $\infty$ is not a Weierstrass point. This implies that \[\ell(4\infty)=4-g+1, \ \ell(5\infty)=5-g+1.\]
    Therefore we conclude that there is a degree $5$ morphism to $\PP^1$ whose polar divisor is equal to $5\infty$.
\end{proof}

\begin{prop}\label{prop125}
The modular curve $X_0(125)$ has a function of degree $5$ over $\Q$.
\end{prop}

\begin{proof}
    We have a degeneracy map from $X_0(125)$ to $X_0(25)$. It is of degree $5$ and is defined over $\Q$ by \cite[Proposition 4.3]{NajmanOrlic22}. The curve $X_0(25)$ is isomorphic to $\PP^1$ since it has genus $0$ and has a rational point. Therefore this degeneracy map defines a function of degree $5$.
\end{proof}

\begin{prop}\label{deg5mapprop}
The modular curve $X_0(N)$ has a function of degree $5$ over $\Q$ for \[N\in\{42,51,52,54-58,63,64,65,67,68,72,73,75,80,91,121\}.\]
\end{prop}

\begin{proof}
    In all these cases, using Magma, we found a rational degree $5$ morphism to $\PP^1$ whose polar divisor is supported on rational cusps and quadratic points.

    We searched for quadratic points as inverse images of rational points on quotient curves $X_0(N)/\left<w_d\right>$. Following that, we computed the Riemann-Roch spaces for all $\Q$-rational effective divisors of degree $5$ supported on the finitely many points of degrees 1 and 2 that we found. Since these divisors are defined over $\Q$, Magma also computes their Riemann-Roch spaces over $\Q$ \cite{magma}, ensuring that the functions found in these Riemann-Roch spaces are also defined over $\Q$. For each curve, inside at least one Riemann-Roch space, we found a function of degree $5$. This strategy for finding degree 5 functions is implemented in the file \github{deg_5_function_search.m}, and the actual search for the aforementioned levels is done in \github{deg_5_functions.m}.
\end{proof}

\section{Modular curves with infinitely many degree 5 points.}\label{deg5pointssection}

\begin{proof}[Proof of \Cref{quinticcor}]
    The levels for which the curve $X_0(N)$ has only finitely many points of degree $\leq4$ have been solved by \Cref{densitydeg5thm}. The remaining levels are 
    \begin{align*}
        N\in\{&1-75,77-83,85-89,91,92,94-96,98-101,103,104,107,111,\\
        &118,119,121,123,125,128,131,141-143,145,155,159,167,191\}.
    \end{align*}

    In \Cref{infinitelydeg5section} we proved that for \[N\in\{1-45,47-58,61,63,64,65,67,68,72,73,75,80,81,91,109,121,125\}\] the curve $X_0(N)$ has infinitely many degree $5$ points. For \[N\in\{46,59,60,62,66,69,70,71,78,87,94,95,104,119\}\] the Jacobian $J_0(N)$ has rank $0$ over $\Q$ by \cite[Theorem 3.1 (1)]{Deg3Class}. Also, \Cref{csprop} proves that there are no degree $5$ morphisms from $X_0(N)$ to $\PP^1$ in these cases. Therefore, by \Cref{translate_of_abelian_variety_thm}, it follows that for these levels $N$, the curve $X_0(N)$ has only finitely many degree $5$ points.
\end{proof}

\subsection{Overview of the open cases}

\Cref{quinticcor} makes significant progress in determining the curves $X_0(N)$ with infinitely many quintic points. However, there are still 30 cases listed in \Cref{rem:remain-cases} that need to be considered. 


The main problem in studying quintic points on these curves is that all these curves have infinitely many quartic points. This makes it harder to study the infinitude of quintic points since it is no longer sufficient to determine whether $\Sym^5 X_0(N)(\Q)$ is infinite or not. This is because $\Sym^5 X_0(N)(\Q)$ will also contain infinitely many points coming from $X_0(N)(\Q) \times \Sym^4 X_0(N)(\Q)$, which are not degree $5$ points. Note that in certain cases, like $N=131$, there are also infinitely many points coming from $\Sym^2 X_0(N)(\Q) \times \Sym^3 X_0(N)(\Q)$ that we need to exclude. This problem did not appear in the determination of the $X_0(N)$ with infinitely many quartic points, since all the curves with infinitely many points of degree $<4$ also have infinitely many degree $4$ points.

Yet another problem caused by the infinitude of the degree $4$ points is that it means that $\min(\delta(X_0(N)/\Q)) \leq 4$. Therefore, we cannot use the results from \cite{KV2025}, since the techniques of that paper only focus on determining the smallest degree for which a curve has infinitely many points of that degree. To be more precise, two of the main tools, \Cref{KVthm} and \Cref{kadetsvogt_prop_DF}, are not available when studying points of degree larger than $\min(\delta(X_0(N)/\Q))$. An extension of the results of \cite{KV2025} to degrees larger than $\min(\delta(X_0(N)/\Q))$ would be very helpful to dealing with the larger genus cases in \Cref{positive_rank_AV_dim12_table}. 

Now we split the discussion into two cases according to the (analytic) rank of $J_0(N)$.

\textbf{1. the rank $0$ case: } \[N\in\{96,98,100\}\]

    For these levels, the rank of $J_0(N)$ is zero by \cite[Theorem 3.1 (1)]{Deg3Class}. Therefore, by \Cref{translate_of_abelian_variety_thm} the curve $X_0(N)$ will have infinitely many quintic points if and only if there exists a degree $5$ morphism to $\PP^1$. In principle, determining the existence of this degree $5$ morphism is a finite computation if we could find a set of generators of $J_0(N)(\Q)$. However, we were not able to find these generators. Additionally, the Castelnuovo-Severi inequality (\Cref{tm:CS}) does not work in these cases. We ran the Magma code that we used to search for degree $5$ functions in \Cref{deg5mapprop}. This search did not find any degree 5 functions, leaving the possibility that no degree $5$ function exists in these $3$ cases.

\textbf{2. the positive rank case: } \begin{align*}
        N\in\{&74,77,79,82,83,85,86,88,89,92,99,101,103,107,111,118,\\
        &123,128,131,141,142,143,145,155,159,167,191\}
    \end{align*}

    For these levels, we have recorded information relevant to the study of the degree $5$ points on the curve $X_0(N)$ in \Cref{positive_rank_AV_dim12_table}. First, we will discuss what impact this information has on determining the degree $5$ points, and a justification for the information in this table is given in \Cref{sec:table-explanation}.

    \Cref{translate_of_abelian_variety_thm} gives two different possibilities for the existence of infinitely many degree $5$ points. Either these points come from a function of degree $5$ or they come from a certain translate of a positive rank abelian variety in $W_5^0 X_0(N).$

    As can be seen from \Cref{positive_rank_AV_dim12_table}, we can rule out the existence of this degree $5$ function in most cases. However, for  \[N\in\{74,77,79,85,88,103\}\] we were not able to prove or disprove the existence of a degree $5$ function. So this is one problem that needs to be solved in these $6$ cases.

    For these levels, the analytic rank of $J_0(N)$ is positive \cite[Theorem 3.1]{Deg3Class}. So, unless we encounter a counterexample to the BSD conjecture, this means that we need to consider the second possibility of  \Cref{translate_of_abelian_variety_thm} regarding translates of positive rank abelian varieties as well. That is, we either must rule out the possibility of $W_5^0 X_0(N)$ containing a translate of a positive rank abelian variety parameterizing infinitely many degree $5$ points, or we need to somehow exhibit a translate of a positive rank abelian variety $W_5^0 X_0(N)$ that gives rise to infinitely many degree $5$ points. 

    If $W_5^0 X_0(N)$ contains a translate of a positive rank abelian variety parametrizing infinitely many degree $5$ points, then it will be difficult to find it. Indeed, \Cref{thm:pos_rank_penta} shows that there are no positive rank pentaelliptic modular curves, so we cannot find it by writing down an explicit map of degree $5$ to an elliptic curve. Another possibility to find this translate is by showing that $X_0(N)$ is a degree $5$ Debarre-Fahlaoui curve. But even if $X_0(N)$ happens to be a degree $5$ Debarre-Fahlaoui curve for some value of $N$, the authors still do not know how to show this. On top of that, there could still be other reasons why $W_5^0 X_0(N)$ contains a translate of a positive rank abelian variety parametrizing infinitely many degree $5$ points, even when $g(X_0(N))>8$ since \Cref{kadetsvogt_prop_DF} is not applicable.

    On the other hand, if $W_5^0 X_0(N)$ does not contain a translate of a positive rank abelian variety parametrizing infinitely many degree $5$ points, then it will also be difficult to show this. The difficulty of this depends on how many simple positive rank abelian varieties are contained in $J_0(N)$.

    If $J_0(N)$ contains a single simple positive rank abelian variety $A$, then there are only finitely many translates of $A$ in $\Pic^5 X_0(N)$. Indeed, every translate of $A$ in $\Pic^5 X_0(N)$ is of the form $5\infty + P + A$ where $P$ is a coset representative in the finite set $J_0(N)(\Q)/A(\Q)$. Therefore, if one can compute representatives of $J_0(N)(\Q)/A(\Q)$, then in theory one could study the different translates of $A$ in $\Pic^5 X_0(N)$ on a case by case basis. So, for each of them one could either try to show that it is not contained in $W_5^0$ using techniques similar to those in \Cref{cor_112_117} or that it comes from adding a rational point on $X_0(N)$ to a translate in $W_4^0$. However, the latter task would also involve actually determining all translates of $A$ in $W_4^0$, which is not an easy task, especially since in most cases, $W_4^0$ will actually contain a translate of $A$ according to the last column of \Cref{positive_rank_AV_dim12_table}.

    If $J_0(N)$ contains multiple simple positive rank abelian varieties $A$, then we are even worse off. Indeed, let $A$ and $A'$ be different, but possibly isogenous, positive rank abelian varieties inside $J_0(N)$. Then there are a lot of cases where we know that a translate $A+P$ of $A$ is contained in $W_4^0 X_0(N)$, as can be seen in \cref{positive_rank_AV_dim12_table}. In these cases, reducing modulo some prime $p$ as in \Cref{positive_rank_AV_dim12_table} is unlikely to yield any useful information. Indeed, in this case $(A+P+\infty)_{\F_p}$ is contained in $W_5^0 X_0(N)_{\F_p}$. So we would need to somehow show that  $(A+P+\infty)$ is the only translate of $A$ contained in $W_5^0 X_0(N)$ among the infinitely many translates of $A$ in $\Pic^5 X_0(N)$ specializing to $(A+P+\infty)_{\F_p}$. 

\begin{center}
\begin{longtable}{|c|c|c|c|c|c|c|} 

\caption{Simple abelian varieties of positive analytic rank inside $J_0(N)$ for the remaining levels. The data explanation is in \Cref{sec:table-explanation}.}\label{positive_rank_AV_dim12_table}\\
\hline
$N$ & genus & deg. $5$ function & $A$ & $\dim$ & multiplicity & $\subseteq W_4^0 X_0(N)$\\
    \hline

$74$  &  8 &  ? &  37.2.a.a & $1$ & $2$ & yes$^*$\\
$77$  &  7 &  ? &  77.2.a.a & $1$ & $1$ & yes$^\dagger$\\
$79$  &  6 &  ? &  79.2.a.a & $1$ & $1$ & yes$^\dagger$\\
$82$  &  9 & no &  82.2.a.a & $1$ & $1$ & yes$^\dagger$\\
$83$  &  7 & no &  83.2.a.a & $1$ & $1$ & yes$^\dagger$\\
$85$  &  7 &  ? &  85.2.a.b & $2$ & $1$ & ?  \\
$86$  & 10 & no &  43.2.a.a & $1$ & $2$ & yes$^*$\\
$88$  &  9 &  ? &  88.2.a.a & $1$ & $1$ & ?  \\
$89$  &  7 & no &  89.2.a.a & $1$ & $1$ & yes$^\dagger$\\
$92$  & 10 & no &  92.2.a.a & $1$ & $1$ & ?  \\
$99$  &  9 & no &  99.2.a.a & $1$ & $1$ & yes$^\dagger$\\
$101$ &  8 & no & 101.2.a.a & $1$ & $1$ & yes$^\dagger$\\
$103$ &  8 &  ? & 103.2.a.a & $2$ & $1$ & yes$^+$\\
$107$ &  9 & no & 107.2.a.a & $2$ & $1$ & yes$^+$\\
$111$ & 11 & no & 37.2.a.a  & $1$ & $2$ & yes$^*$\\
$118$ & 14 & no & 118.2.a.a & $1$ & $1$ & yes$^\dagger$\\
$123$ & 13 & no & 123.2.a.a & $1$ & $1$ & ?  \\
      &    &    & 123.2.a.b & $1$ & $1$ & yes$^\dagger$\\
$128$ &  9 & no & 128.2.a.a & $1$ & $1$ & yes$^\dagger$\\
$131$ & 11 & no & 131.2.a.a & $1$ & $1$ & yes$^\dagger$\\
$141$ & 15 & no & 141.2.a.a & $1$ & $1$ & ?  \\
      &    &    & 141.2.a.d & $1$ & $1$ & yes$^\dagger$\\
$142$ & 17 & no & 141.2.a.a & $1$ & $1$ & yes$^\dagger$\\
      &    &    & 141.2.a.d & $1$ & $1$ & ?  \\
$143$ & 13 & no & 143.2.a.a & $1$ & $1$ & yes$^\dagger$\\
$145$ & 13 & no & 145.2.a.a & $1$ & $1$ & yes$^\dagger$\\
      &    &    & 145.2.a.b & $2$ & $1$ & ?  \\
$155$ & 15 & no & 155.2.a.a & $1$ & $1$ & ?  \\
      &    &    & 155.2.a.c & $1$ & $1$ & yes$^\dagger$\\
$159$ & 17 & no & 53.2.a.a  & $1$ & $2$ & yes$^*$\\
$167$ & 14 & no & 167.2.a.a & $2$ & $1$ & yes$^+$\\
$191$ & 16 & no & 191.2.a.a & $2$ & $1$ & yes$^+$\\

\hline
\end{longtable}
\end{center}

\begin{remark}
    All abelian varieties $A$ in the table have analytic rank equal to their dimension. We deduced this from the analytic ranks of the newform data in the LMFDB. If $A$ is an elliptic curve, then this is because $L(A,s)=L(f,s)$, where $f$ is the newform corresponding to $A$. If $A$ has dimension $2$, then it corresponds to a Galois orbit of two Galois conjugate newforms $f$ and $f'$. The L-function of $A$ is the product $L(f,s)L(f',s)$.
\end{remark}

\subsection{Justification for \texorpdfstring{\Cref{positive_rank_AV_dim12_table}}{the open case Table}.}
\label{sec:table-explanation}
Below we describe how we obtained the information in the different columns of \Cref{positive_rank_AV_dim12_table}.

    {\bf deg. 5 function.} This column records whether $X_0(N)$ has a function of degree $5$ over $\Q$. As we can see, for some levels in the table, we were able to prove that there are no such functions. For \[N\in\{83,89,92,101,107,111,131,141,142,143,167,191\}\] this is due to \Cref{csprop} and for \[N\in\{82,86,99,118,123,128,145,155,159\}\] this is because the $\Q$-gonality of $X_0(N)$ is at least $6$. However, for \[N\in\{74,77,79,85,88,103\}\] the $\Q$-gonality is equal to $4$ and Castelnuovo-Severi inequality does not give a contradiction, so we do not know whether a degree $5$ function exists. The Magma code that we used to search for a degree 5 function in \Cref{deg5mapprop} did not find such a function in all these cases. This suggests that there might not be a degree 5 function in at least some of these cases.

    {\bf $\bm{A}$, dim, and multiplicity.} The information in these columns comes from the LMFDB. In the $A$ column, we list the isogeny classes of simple abelian varieties $A$ with positive analytic rank  inside $J_0(N)$. These abelian varieties are given by the LMFDB label of the corresponding Galois orbit of newforms. The \textit{dim} column records the dimension of $A$. The \textit{multiplicity} column lists the largest integer $n$ such that $A^n$ is an isogeny factor of $J_0(N)$. This integer is just the number of divisors of $N/level(A)$, as follows from \cite[Theorem 9.4]{stein07}.

    {\bf $\bm{\subseteq W_4^0 X_0(N)}$.} In this column, we record whether there exists an abelian variety $A'$ in the isogeny class of $A$ that is contained in $W_4^0 X_0(N)$. In the cases where we know this, it is due to one of the following reasons.
    \begin{itemize}
        \item $A$ is an elliptic curve of conductor $N$ and modular degree $\leq 4$. In this case we get an $A' \subset W_4^0 X_0(N)$ in the isogeny class of $A$ because we have a degree $4$ map $X_0(N) \to A$. The cases where this argument applies have been labeled by a $^\dagger$.
        \item $A$ is an elliptic curve of conductor $<N$ occurring with multiplicity $2$, $X_0(N) \to X_0^*(N)$ is of degree $4$, and $A$ is isogenous to $J^*_0(N)$. In this case we get an $A' \subset W_4^0 X_0(N)$ in the isogeny class of $A$ because we again have a degree $4$ map $X_0(N) \to A$. The cases where this argument applies have been labeled by a $^*$.
        \item $A$ has dimension $2$ and is isogenous to $J_0^+(N)$. In this case $X_0(N)^+$ is of genus $2$ so that $W_2 X_0(N)^+ = \Pic^2 X_0(N)^+$. Pulling back along the morphism $f:X_0(N) \to X_0(N)^+$ of degree two gives $f^*(\Pic^2 X_0(N)^+) = f^*(W_2 X_0(N)^+) \subseteq W_4 X_0(N)^+$. The cases where this argument applies have been labeled by a $^+$.
    \end{itemize}
\bibliographystyle{siam}
\bibliography{bibliography1}

\def\cprime{$'$} \def\cprime{$'$}
\begin{thebibliography}{10}

\bibitem{abramovich}
{\sc D.~Abramovich}, {\em A linear lower bound on the gonality of modular
  curves}, Internat. Math. Res. Notices,  (1996), pp.~1005--1011.

\bibitem{AbramovichHarris91}
{\sc D.~{Abramovich} and J.~{Harris}}, {\em {Abelian varieties and curves in
  \(W_ d (C)\)}}, {Compos. Math.}, 78 (1991), pp.~227--238.

\bibitem{akmnov}
{\sc N.~Ad{\v{z}}aga, T.~Keller, P.~Michaud-Jacobs, F.~Najman, E.~Ozman, and
  B.~Vukorepa}, {\em Computing quadratic points on modular curves
  {{\(X_0(N)\)}}}, Math. Comp., 93 (2024), pp.~1371--1397.

\bibitem{BGGP}
{\sc M.~H. Baker, E.~Gonz{\'a}lez-Jim{\'e}nez, J.~Gonz{\'a}les, and B.~Poonen},
  {\em Finiteness results for modular curves of genus at least 2}, Am. J.
  Math., 127 (2005), pp.~1325--1387.

\bibitem{Bars99}
{\sc F.~Bars}, {\em Bielliptic modular curves}, J. Number Theory, 76 (1999),
  pp.~154--165.

\bibitem{magma}
{\sc W.~Bosma, J.~Cannon, and C.~Playoust}, {\em The {Magma} algebra system.
  {I}: {The} user language}, J. Symb. Comput., 24 (1997), pp.~235--265.

\bibitem{BELOV}
{\sc A.~Bourdon, O.~Ejder, Y.~Liu, F.~Odumodu, and B.~Viray}, {\em On the level
  of modular curves that give rise to isolated {$j$}-invariants}, Adv. Math.,
  357 (2019), pp.~106824, 33.

\bibitem{DebarreFahlaoui}
{\sc O.~Debarre and R.~Fahlaoui}, {\em Abelian varieties in {$W^r_d(C)$} and
  points of bounded degree on algebraic curves}, Compositio Math., 88 (1993),
  pp.~235--249.

\bibitem{derickx2024primitivepoints}
{\sc M.~Derickx}, {\em Large degree primitive points on curves}.
\newblock \url{https://arxiv.org/abs/2409.05796}, 2024.

\bibitem{Deg3Class}
{\sc M.~{Derickx}, A.~{Etropolski}, M.~{van Hoeij}, J.~S. {Morrow}, and
  D.~{Zureick-Brown}}, {\em {Sporadic cubic torsion}}, {Algebra Number Theory},
  15 (2021), pp.~1837--1864.

\bibitem{DerickxOrlic23}
{\sc M.~Derickx and P.~Orli{\'{c}}}, {\em {Modular curves $X_0(N)$ with
  infinitely many quartic points}}, Res. Number Theory, 10 (2024).

\bibitem{DerickxSutherland17}
{\sc M.~Derickx and A.~V. Sutherland}, {\em Torsion subgroups of elliptic
  curves over quintic and sextic number fields}, Proc. Am. Math. Soc., 145
  (2017), pp.~4233--4245.

\bibitem{derickxVH}
{\sc M.~Derickx and M.~van Hoeij}, {\em Gonality of the modular curve
  {$X_1(N)$}}, J. Algebra, 417 (2014), pp.~52--71.

\bibitem{frey}
{\sc G.~Frey}, {\em Curves with infinitely many points of fixed degree}, Israel
  J. Math., 85 (1994), pp.~79--83.

\bibitem{HarrisSilverman91}
{\sc J.~Harris and J.~H. Silverman}, {\em Bielliptic curves and symmetric
  products}, Proceedings of the American Mathematical Society, 112 (1991),
  pp.~347--356.

\bibitem{hartshorne1977algebraic}
{\sc R.~Hartshorne}, {\em Algebraic Geometry}, Graduate Texts in Mathematics,
  Springer, 1977.

\bibitem{Hwang2023}
{\sc W.~Hwang and D.~Jeon}, {\em Modular curves with infinitely many quartic
  points}, Math. Comp., 93 (2024), pp.~383--395.

\bibitem{Jeon2021}
{\sc D.~Jeon}, {\em Modular curves with infinitely many cubic points}, J.
  Number Theory, 219 (2021), pp.~344--355.

\bibitem{Jeon2022}
{\sc D.~Jeon}, {\em {Trielliptic modular curves \(X_1(N)\)}}, Acta Arith., 206
  (2022), pp.~171--188.

\bibitem{Jeon2025}
{\sc D.~Jeon}, {\em Tetraelliptic modular curves {{\(X_1(N)\)}}}, Acta Arith.,
  219 (2025), pp.~33--51.

\bibitem{JeonKimPark06}
{\sc D.~Jeon, C.~H. Kim, and E.~Park}, {\em On the torsion of elliptic curves
  over quartic number fields}, J. London Math. Soc. (2), 74 (2006), pp.~1--12.

\bibitem{JeonKimSchweizer04}
{\sc D.~Jeon, C.~H. Kim, and A.~Schweizer}, {\em On the torsion of elliptic
  curves over cubic number fields}, Acta Arith., 113 (2004), pp.~291--301.

\bibitem{KV2025}
{\sc B.~Kadets and I.~Vogt}, {\em Subspace configurations and low degree points
  on curves}, Adv. Math., 460 (2025), p.~36.
\newblock Id/No 110021.

\bibitem{kamienny92}
{\sc S.~Kamienny}, {\em Torsion points on elliptic curves and
  {$q$}-coefficients of modular forms}, Invent. Math., 109 (1992),
  pp.~221--229.

\bibitem{kenku1979}
{\sc M.~A. {Kenku}}, {\em {The modular curve \(X_0(39)\) and rational
  isogeny}}, Math. Proc. Cambr. Phil. Soc., 85 (1979), p.~21–23.

\bibitem{kenku1980_2}
\leavevmode\vrule height 2pt depth -1.6pt width 23pt, {\em {The modular curve
  \(X_0(169)\) and rational isogeny}}, J. Lond. Math. Soc., s2-22 (1980),
  pp.~239--244.

\bibitem{kenku1980_1}
\leavevmode\vrule height 2pt depth -1.6pt width 23pt, {\em {The modular curves
  \(X_0(65)\) and \(X_0(91)\) and rational isogeny}}, Math. Proc. Cambr. Phil.
  Soc., 87 (1980), p.~15–20.

\bibitem{kenku1981}
\leavevmode\vrule height 2pt depth -1.6pt width 23pt, {\em {On the modular
  curves \(X_0(125)\), \(X_1(25)\) and \(X_1(49)\)}}, {J. Lond. Math. Soc., II.
  Ser.}, 23 (1981), pp.~415--427.

\bibitem{KM88}
{\sc M.~A. Kenku and F.~Momose}, {\em Torsion points on elliptic curves defined
  over quadratic fields}, Nagoya Math. J., 109 (1988), pp.~125--149.

\bibitem{Kim2002}
{\sc H.~Kim}, {\em {Functoriality for the exterior square of $\textup{GL}_4$
  and the symmetric fourth of $\textup{GL}_2$. Appendix 2: Refined estimates
  towards the Ramanujan and Selberg conjectures (by Henry Kim and Peter
  Sarnak)}}, J. Amer. Math. Soc., 16 (2002), p.~139–183.

\bibitem{lmfdb}
{\sc {LMFDB Collaboration}}, {\em The {L}-functions and modular forms
  database}.
\newblock \url{http://www.lmfdb.org}, 2025.
\newblock [Online; accessed 13 February 2025].

\bibitem{maple}
{\sc {Maplesoft, a division of Waterloo Maple Inc.}}, {\em Maple}.

\bibitem{mazur77}
{\sc B.~Mazur}, {\em Modular curves and the {E}isenstein ideal}, Inst. Hautes
  \'Etudes Sci. Publ. Math.,  (1977), pp.~33--186 (1978).

\bibitem{mazur78}
\leavevmode\vrule height 2pt depth -1.6pt width 23pt, {\em Rational isogenies
  of prime degree (with an appendix by {D}. {G}oldfeld)}, Invent. Math., 44
  (1978), pp.~129--162.

\bibitem{NajmanOrlic22}
{\sc F.~Najman and P.~Orlić}, {\em {Gonality of the modular curve \(X_0(N)\)
  }}, Math. Comp., 93 (2023), p.~863–886.

\bibitem{NguyenSaito}
{\sc K.~V. Nguyen and M.-H. Saito}, {\em $d$-gonality of modular curves and
  bounding torsions}.
\newblock \url{https://arxiv.org/alg-geom/9603024}, 1996.

\bibitem{Ogg74}
{\sc A.~P. Ogg}, {\em Hyperelliptic modular curves}, Bull. Soc. Math. France,
  102 (1974), pp.~449--462.

\bibitem{selberg}
{\sc A.~Selberg}, {\em On the estimation of Fourier coefficients of modular
  forms}, vol.~8, Proceedings of Symposia in Pure Mathematics, American
  Mathematical Society, 1965, pp.~1--15.

\bibitem{serre2016topics}
{\sc J.-P. Serre}, {\em Topics in {Galois} theory. {Notes} written by {Henri}
  {Darmon}}, vol.~1 of Res. Notes Math., Wellesley, MA: A K Peters, 2nd~ed.,
  2007.

\bibitem{silverman}
{\sc J.~H. Silverman}, {\em The arithmetic of elliptic curves}, vol.~106 of
  Graduate Texts in Mathematics, Springer, Dordrecht, 2nd~ed., 2009.

\bibitem{stein07}
{\sc W.~Stein}, {\em Modular forms, a computational approach. {With} an
  appendix by {Paul} {E}. {Gunnells}}, vol.~79 of Grad. Stud. Math.,
  Providence, RI: American Mathematical Society (AMS), 2007.

\bibitem{Stichtenoth09}
{\sc H.~Stichtenoth}, {\em Algebraic function fields and codes}, vol.~254 of
  Grad. Texts Math., Berlin: Springer, 2nd ed.~ed., 2009.

\bibitem{PARIGP}
{\sc {The PARI~Group}}, {\em User's Guide to {PARI-GP} \texttt{2.17.1}}, Univ.
  Bordeaux,
  \url{https://pari.math.u-bordeaux.fr/pub/pari/manuals/2.17.1/users.pdf},
  2024.
\newblock see also \url{https://pari.math.u-bordeaux.fr/}.

\bibitem{sagemath}
{\sc {The Sage Developers}}, {\em {S}age{M}ath, the {S}age {M}athematics
  {S}oftware {S}ystem}, 2022.
\newblock DOI 10.5281/zenodo.6259615.

\end{thebibliography}

\end{document}